\documentclass[reqno,12pt]{amsart}
\usepackage{graphicx} % Required for inserting images
\usepackage[margin=1.1in]{geometry}

\usepackage[utf8]{inputenc}

\usepackage{comment}

\usepackage{amsmath, amsfonts, amssymb, tikz, amsthm, mathtools}
\usepackage{longtable}
\usepackage{caption}
\usetikzlibrary{calc}
\usepackage[pdfencoding=auto,psdextra]{hyperref}
\usepackage{dsfont}
\usepackage{mathabx,epsfig}
\usepackage{ytableau}
\usepackage{bbold}
\usepackage{faktor}
\usepackage{multicol}
\usepackage{wrapfig}
\usepackage{enumitem}
\usepackage{bm} %for bold math letters
\usepackage{wrapfig}
\usepackage{float}
\usepackage{ifthen}
\usepackage[
  backend=biber,
  isbn=false,
  url=false,
  sorting=nyt,
  maxbibnames=8,
  doi=false
]{biblatex}
\AtBeginBibliography{\small}

\definecolor{light-gray}{gray}{0.8}
\counterwithin{equation}{section}

{\theoremstyle{plain}
\newtheorem{introtheorem}{Theorem}
\newtheorem{theorem}{Theorem}[section]
\newtheorem{conjecture}[theorem]{Conjecture}
\newtheorem{proposition}[theorem]{Proposition}
\newtheorem{lemma}[theorem]{Lemma}
\newtheorem{corollary}[theorem]{Corollary}
\newtheorem{question}[theorem]{Question}
}
{\theoremstyle{definition}
\newtheorem{definition}[theorem]{Definition}
\newtheorem*{definition*}{Definition}
\newtheorem{remark}[theorem]{Remark}
\newtheorem*{remark*}{Remark}

\newtheorem{example}[theorem]{Example}
}

\newcommand{\cl}{\mathcal{L}}
\newcommand{\cd}{\mathcal{D}}

\newcommand{\Wn}{\mathfrak{W}_n}
\newcommand{\Sn}{\mathfrak{S}_n}

\newcommand{\Wi}{\mathfrak{W}^i}

\newcommand{\C}{\mathbb{C}}   % complex numbers.
\newcommand{\Z}{\mathbb{Z}}   % integers
\newcommand{\ds}{\displaystyle}
\newcommand{\abs}[1]{\left| #1 \right|}
\newcommand{\llangle}{\left\langle}
\newcommand{\rrangle}{\right\rangle}
\newcommand{\polyr}[1]{\C[x_1,\ldots,x_{#1}]}
\newcommand{\poly}{\C[x_\bullet]}
\newcommand{\hdes}[1]{\mathcal{D}_H\!\left(#1\right)}

\newcommand{\bn}{\overline{n}}
\newcommand{\bi}{\overline{i}}
\newcommand{\bj}{\overline{j}}

\renewcommand{\bar}[1]{\overline{#1}}
\newcommand{\fh}{\mathfrak{h}}
\newcommand{\fs}{\mathfrak{s}}
\newcommand{\bchi}{\bm{\chi}}
\newcommand{\triv}{\mathbb{1}}

 \newcommand{\splines}[1]{\mathcal{M}_{#1}}
 \newcommand{\lqot}[1]{\mathrm{L}_{#1}}
 \newcommand{\rqot}[1]{\mathrm{R}_{#1}}
 \newcommand{\lrep}[1]{\mathfrak{ch}\left(\mathrm{L}_{#1}\right)}
 \newcommand{\rrep}[1]{\mathfrak{ch}\left(\mathrm{R}_{#1}\right)}
 \newcommand{\choos}[2]{\genfrac(){0pt}{0}{#1}{#2}}

 \DeclareMathOperator{\Neg}{Ng}
 \DeclareMathOperator{\Hess}{Hess}
 
 \DeclareMathOperator{\Frob}{Frob}

\newcommand{\spline}[1]{\bm{#1}}
\title{Signed permutations and degree-one dot action representations for types B and C}
\author{Nathan R.\,T. Lesnevich}
\thanks{The author was partially supported by NSF grant DMS-1954001 during the initial phases of this project. The author thanks Hsin-Chieh Liao for helpful conversations, and would like to thank Martha Precup and John Shareshian for their continued advice and support.}
\address{Department of Mathematics, Oklahoma State University, Stillwater, OK 74078}
\email{\href{mailto:nlesnev@okstate.edu}{nlesnev@okstate.edu}}

\begin{document}

\begin{abstract}
     %[For geom. talks] 
     A spline is an assignment of polynomials to the vertices of a graph, where the difference of two polynomials along an edge must belong to the ideal labeling that edge. We consider a ring of splines $\splines{H}$ constructed on a graph whose vertices are the Weyl group $\Wn$ of signed permutations, and whose edges and edge-ideals are defined using an order ideal $H$ of positive roots. These splines are a module over the polynomial ring in two ways, and a $\Wn$-module by the dot action. These structures on $\splines{H}$ give rise to the graded left and right dot action representations of $\Wn$. The left representation is the type B/C generalization of the type A dot action for regular semisimple Hessenberg varieties (and thus, chromatic quasisymmetric functions), and the right representation is the same for corresponding manifolds of isospectral matrices (and thus, unicellular LLT polynomials). This paper gives explicit module generators for the degree-one graded piece of $\splines{H}$ and computes the degree-one piece of the both dot action representations for all $H$ using the combinatorial data of $H$.
\end{abstract}
\maketitle

\setcounter{tocdepth}{1}
%\tableofcontents

\section{Introduction}\label{sec:intro}
Let $\Wn$ be the group of signed permutations, which are permutations $w$ on the set $\{1,\ldots,n,-n,\ldots,-1\}$ such that $w(-i) = -w(i)$. The group $\Wn$ is also the Weyl group of types B and C, and so is associated with two distinct root systems. A Hessenberg space $H$ is a particular subset of roots, which determines a set $S(H)$ of transpositions in $\Wn$. Given a Hessenberg space, one can define the \emph{graded ring of splines} $\splines{H}$ on $\Wn$. This ring has two $\poly\coloneqq \polyr{n}$-module structures and a $\Wn$-module structure. This paper determines the algebraic and combinatorial properties of the degree-one graded piece of $\splines{H}$ from the combinatorics of $S(H)$.\par 
Let $(i,j)$ be the unique transposition in $\Wn$ that switches $i$ with $j$ and $-i$ with $-j$ (note $i=-j$ is possible). If $H$ is a Hessenberg space, then the ring of splines is 
 \[
        \splines{H} = \bigoplus_{i \geq 0} \splines{H}^i \coloneqq \left\{ \left.\spline{\rho} \in \prod_{w \in \Wn} \poly\; \right|\; \spline{\rho}(w) - \spline{\rho}(w(i,j)) \in \llangle x_{w(i)}-x_{w(j)} \rrangle\text{ if } (i,j) \in S(H) \right\}\!,
   \]
with (graded) $\Wn$-module structure $w\cdot \spline{\rho}(v) = w\spline{\rho}(w^{-1}v)$ and (graded) $\poly$-module structure given by multiplication.  \par 
The ring $\splines{H}$ is isomorphic to the $T$-equivariant cohomology of a smooth subvariety of the full flag variety $G/B$ called a \emph{regular semisimple Hessenberg variety}. The connection between this combinatorial ring and equivariant cohomology is called \emph{GKM Theory} \cite{GKM_theory}, and it allows for the study of this cohomology in purely combinatorial and algebraic terms. \par 
The $\Wn$-module structure on $\splines{H}$ was first defined as the \emph{dot action} on equivariant cohomology by Tymoczko in \cite{tymoczko2008permutation}. There are two natural $\Wn$-equivariant quotients, called the left $\lqot{H}$ and right $\rqot{H}$ quotients, of $\splines{H}$ that are in fact graded $\C$-vector spaces (the left quotient corresponds to ordinary cohomology). The graded $\Wn$-module structure of $\splines{H}$ induces graded $\Wn$-representations on the quotients. The primary goal of this paper is to compute the characters of the degree-one pieces of 
\[
\lqot{H} = \bigoplus_{i\geq 0} \left(\lqot{H}\right)_i \;\; \text{ and }\;\;\rqot{H} = \bigoplus_{i\geq 0} \left(\rqot{H}\right)_i
\] from the combinatorial data of $S(H)$.  In doing so, we compute the dot action on second (equivariant and ordinary) cohomology for all regular semisimple Hessenberg varieties in types B and C. This generalizes what is known in type A \cite{chow_linearp,chohonglee_second_cohom,ayzenberg2022second}.\par 
These graded representations are of interest to algebraic combinatorists in part because they are the $\Wn$-equivalents of very well-studied $\Sn$-representations, which have connections to chromatic symmetric functions \cite{guaypaquet2016shar_wachs_conj,SW2016chromaticquasisymmetric,brosnan_chow_dotactn_is_chromsym} and LLT polynomials \cite{guaypaquet2016shar_wachs_conj,Ayzenberg2018isospectral, ALEXANDERSSON2018LLTchromsym}. We detail some of these connections in the Appendix. From these works, in type A there are combinatorial formulas to compute the characters of both $\lqot{H}$ and $\rqot{H}$ in at least one basis for the character space of $\Sn$. This is \textbf{not} the case in types B or C.\par 
For one specific family of $H$ (those corresponding to simple roots), the character of $\lqot{H}$ was computed by Stembridge \cite{Stembridge_Permuto} in general type. Stembridge's results are via a connection to \emph{permutohedra}, as the associated Hessenberg varieties are in fact toric varieties. While it is possible to compute the left and right characters for general $H$ in types B and C, even partial results are rare and rely on sophisticated geometric tools \cite{BalibanuCrooks_sheaves,PrecupSommers_sheaves}. Formulas are non-combinatorial (and usually non-positive), require the computation of type B/C Green polynomials and Poincar\'e polynomials of other varieties, and proofs require the geometry of perverse sheaves. \par 
In contrast, for the degree-one graded piece, the main result of this paper applies to all Hessenberg spaces $H$ and is entirely computable from the combinatorics of roots and transpositions in $\Wn$. Moreover, we compute the degree-one pieces of the characters $\lrep{H}$ and $\rrep{H}$ in a positive manner.
\begin{introtheorem}\label{intthm:character}
    If $H$ is a type B or type C Hessenberg space, there exist non-negative integers $a,b \in \mathbb{N}$, integers $c,d \in \{0,1\}$, and a subset $I \subseteq [n]$, each determined by $S(H)$, such that
    \[
     \lrep{H}_1 = a\,\mathbb{1} + \sum_{i\in I} \fh_i + b\,\fh_1 + c\,\fs + d\,\bm{\delta} 
    \]
    and 
    \[
     \rrep{H}_1 = \bm{\chi} + \sum_{i\in I} \left(\fh_i - \mathbb{1}\right) + b\left(\fh_1-\mathbb{1}\right) + c \left(
    \fs-\mathbb{1}\right) + d\,\bm{\delta},
    \]
    where 
    \begin{itemize}
        \item $\bm{\chi}$ is the character of the defining representation of $\Wn$ on $\C^n$
        \item $\fh_i$ is the character of the action of $\Wn$ on the cosets of $\mathfrak{S}_i\times \mathfrak{W}_{n-i}$,
        \item $\fs$ is the character of the action of $\Wn$ on the cosets of $\mathfrak{W}_1\times \mathfrak{W}_{n-1}$,
        \item $\mathbb{1}$ is the character of the trivial representation, and 
        \item $\bm{\delta}$ is the character $w\mapsto (-1)^{\abs{\Neg(w)}}$ where $\Neg(w) = \{w(i) \mid i \in [n],\; w(i) < 0\}$.
    \end{itemize}
    Moreover, $c$ and $d$ are never simultaneously $1$, and $d$ is always $0$ in type B.
\end{introtheorem}
Theorem \ref{intthm:character} is Proposition \ref{prop:permutohedral} if $H = \Delta$ and Theorem \ref{thm:character} otherwise, which use language from Definition \ref{def:itypes} to describe the coefficients $a,b,c$, and $d$. The representation ring of $\Wn$ is isomorphic to the \emph{type B/C symmetric functions} via the Frobenius characteristic map \cite{Macdonald_symfuncs}. The tools to translate Theorem \ref{intthm:character} into type B/C symmetric functions are in the Appendix. There, we make the observation that $\lrep{H}_1$ is in fact $h_{\lambda,\mu}$-positive. This is particularly interesting, as it generalizes the behavior seen in type A, where the left graded character is conjectured to be $h_\lambda$-positive and is intimately connected to \emph{chromatic symmetric functions}. This is known as the graded Stanley-Stembridge conjecture \cite{StanleyStembridge,SW2016chromaticquasisymmetric}, and has been the subject of much research \cite{STANLEY1995chromsym, Gasharov96, GuayPaquet13,  harada2017cohomology, brosnan_chow_dotactn_is_chromsym,Dahlberg19, Abreu_Nigro20}. A proof of the ungraded conjecture was recently given by Hikita \cite{Hikita_stanstem}.\par

Our theorem is the first type B/C result to provide an explicit non-negative expansion for a graded piece of $\lrep{H}$ or $\rrep{H}$ in terms of characters of irreducible- and/or permutation-representations for all Hessenberg spaces. It is also the first to compute the characters for any graded piece of $\lrep{H}$ and $\rrep{H}$ by directly providing a basis for the representation, and in doing so gives a family of characters with which one can attempt to expand the catalog of results and connections that exist in type A to types B and C.\par 

The paper is structured as follows. Section \ref{sec:background} provides some of the necessary background on signed permutations, Hessenberg spaces, and splines. Section \ref{sec:ideals_of_transpositions} shows how to translate from Hessenberg spaces $H$ to sets $S(H)$ of signed transpositions in a manner that unifies some type B and type C calculations. Section \ref{sec:trivial} reduces the $S(H)$ that one must consider in order to compute $\splines{H}^1$ to a much smaller collection, one that completely unifies the type B and type C calculations. Section \ref{sec:one_inversion} determines which elements of $\Wn$ have exactly one $H$-inversion for each $H$. Section \ref{sec:linear_splines} defines several sets of splines contained in $\splines{H}^1$, provides linear relations between them, and computes the dot action on them. Section \ref{sec:generators} argues, using the elements from \S \ref{sec:one_inversion}, that the splines from \S \ref{sec:linear_splines} form a $\C$-generating set for $\splines{H}^1$, and uses the linear relations between them to reduce the size of this generating set. Section \ref{sec:bases_reps} reduces this generating set further in two different ways, one that results in a basis for $\splines{H}^1$ that is conducive to computing $\lrep{H}_1$ and one that is conducive to computing $\rrep{H}_1$. Then, we prove the main theorem. Appendix \ref{sec:symmetricfunctions} describes how one can translate this result to the language of type B/C symmetric functions.

\section{Background}\label{sec:background}
 Let $[n] \coloneqq \{1,\ldots ,n\}$ and $[\bn] \coloneqq \{1,\ldots ,n,-n,\ldots ,-1\}$. From here on, we use bar notation for negative integers, so that $\bi \coloneqq -i$. When $n$ is clear, let $\poly \coloneqq \polyr{n}$. \par 
 There are several ways to characterize the group $\Wn$. We will employ two of them in this paper\footnote{Two other commonly used characterizations are as the wreath product of $\mathcal{S}_n$ and $\Z_2$, which useful for describing the representation theory of $\Wn$, and as the symmetry group of a hypercube, so $\Wn$ is frequently called the \emph{hyperoctahedral group}.}. First, described in Subsection \ref{ssec:signedperms} is the characterization taken as the definition, the group of signed permutations. Second, as described in Section \ref{ssec:rootsystems}, $\Wn$ is the Weyl group of the root systems of type B and C. We recommend \cite{HumphreysReflections,BjornerBrentiCoxeter}, for a detailed description of root systems and Weyl groups.
 \subsection{Signed Permutations}\label{ssec:signedperms}
 The group of \emph{signed permutations} $\Wn$ is the group of bijections $w \colon [\bn] \to [\bn]$ such that $w\left(\bi\right) = \overline{w(i)}$ for all $i \in [\bn]$. The one-line notation of $w$ is $[w(1),w(2),\ldots ,w(n)]$. The values of $w\left(\bi\right)$ for $i \in [n]$ are omitted in this notation since they are equal to $\overline{w(i)}$. The elements of $\mathfrak{W}_2$ are listed below in one-line notation.
 \[
 \begin{matrix}
     [1,2] & [1,\overline{2}] & [\overline{1},2] & [\overline{1},\overline{2}] &
     [2,1] & [2,\overline{1}] & [\overline{2},1] & [\overline{2},\overline{1}]
 \end{matrix}
 \]
 The cycle notation of a permutation $w \in \Wn$ is the same as the cycle notation treating $w$ as an element of the symmetric group on letters $[\bn]$. For example, $[\overline{1},2] = (\overline{1},1)$ and $[2,1] = (1,2)(\overline{1},\overline{2})$. A signed permutation is a \emph{transposition} if it is one of the following:
 \begin{itemize}
     \item $(i,j) \coloneqq (i,j)(\overline{i},\overline{j})$,
     \item $(i,\overline{j}) \coloneqq (i,\overline{j})(\overline{i},j)$, or
     \item $(i,\overline{i})$,
 \end{itemize}for $i<j \in [n]$. The cardinality of $\Wn$ is $2^nn!$. \par 
The group $\Wn$ is a Coxeter group with generators $s_1,\ldots ,s_{n-1},s_n$ subject to the following four relations:
\begin{enumerate}
    \item $s_i^2 = e$
    \item $s_is_j = s_js_i$ if $\abs{i-j} > 1$
    \item $s_is_{i+1}s_i = s_{i+1}s_is_{i+1}$ if $i \in [n-2]$
    \item $s_{n-1}s_ns_{n-1}s_n = s_ns_{n-1}s_ns_{n-1}$.
\end{enumerate}
Where $s_i \coloneqq (i,i+1)(\overline{i},\overline{i+1})$ and $s_n$ = $(n,\bn)$. These generators are called \emph{simple transpositions}, relation (2) is called a \emph{commuting move}, and relations (3), (4) are called \emph{braid moves}. Given a signed permutation in one-line notation, right multiplication by $s_i$ for $i < n$ switches the entries in the $i$-th and $i+1$-th positions of one-line notation, and right multiplication by $s_n$ negates the last entry. Similarly, left multiplication by $s_i$ for $i < n$ switches the numbers $i$ and $i+1$ in the one-line notation, and left multiplication by $s_n$ negates $n$. As with all finite Coxeter groups, elements of $\Wn$ have \emph{reduced words} (a minimal-length way of writing $w$ in terms of $s_i$) and the \emph{length} $\ell(w)$ is the length of any reduced word. All reduced words for a particular $w \in \Wn$ are related by a sequence of commuting and/or braid moves. \par 
\begin{definition}\label{def:Bruhat}
    The \emph{Bruhat order} on $\Wn$ is the transitive closure of the relations $wt < w$ if $t$ is a transposition and $\ell(wt) < \ell(w)$. 
\end{definition}
It is well known that the Bruhat order is equivalent to the order $v<w$ if a reduced word for $v$ appears as a subword of a reduced word for $w$.
\begin{example}\label{ex:bruhat}
    The following two figures are Hasse diagrams for the Bruhat order on $\mathfrak{W}_2$. One shows the one-line notation (to emphasize the transpositions), and the other shows the reduced word (to emphasize the subwords).
    
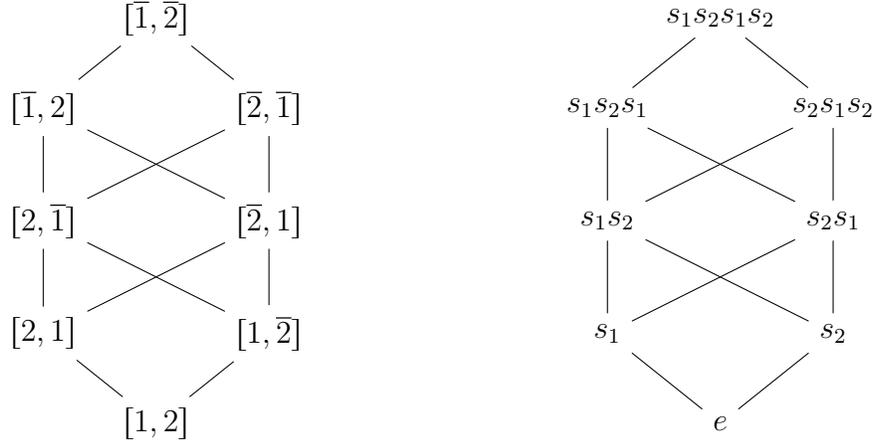
\begin{center}
        \begin{tikzpicture}[scale=1.5]
        \begin{scope}
            \draw (0,0.2) node (e) {$[1,2]$};
            \draw (-1,1) node (l1) {$[2,1]$};
            \draw (-1,2) node (l2) {$[2,\bar{1}]$};
            \draw (-1,3) node (l3) {$[\bar{1},2]$};
            \draw (1,1) node (r1) {$[1,\bar{2}]$};
            \draw (1,2) node (r2) {$[\bar{2},1]$};
            \draw (1,3) node (r3) {$[\bar{2},\bar{1}]$};
            \draw (0,3.8) node (w0) {$[\bar{1},\bar{2}]$};

            \draw (e)--(l1)--(l2)--(l3)--(w0);
            \draw (e)--(r1)--(r2)--(r3)--(w0);
            \draw (l1)--(r2);
            \draw (l2)--(r3);
            \draw (r1)--(l2);
            \draw (r2)--(l3);
        \end{scope}
        \begin{scope}[xshift=5cm]
            \draw (0,0.2) node (e) {$e$};
            \draw (-1,1) node (l1) {$s_1$};
            \draw (-1,2) node (l2) {$s_1s_2$};
            \draw (-1,3) node (l3) {$s_1s_2s_1$};
            \draw (1,1) node (r1) {$s_2$};
            \draw (1,2) node (r2) {$s_2s_1$};
            \draw (1,3) node (r3) {$s_2s_1s_2$};
            \draw (0,3.8) node (w0) {$s_1s_2s_1s_2$};

            \draw (e)--(l1)--(l2)--(l3)--(w0);
            \draw (e)--(r1)--(r2)--(r3)--(w0);
            \draw (l1)--(r2);
            \draw (l2)--(r3);
            \draw (r1)--(l2);
            \draw (r2)--(l3);
        \end{scope}
        \end{tikzpicture}
    \end{center}
\end{example}
\begin{definition}\label{def:Young}
    Given a subset $J \subseteq \{s_1,\ldots ,s_n\}$ of simple reflections, the \emph{Young subgroup} $\mathfrak{W}_J$ is the subgroup of $\Wn$ generated by $J$. The (left) cosets of $\mathfrak{W}_J$ have a unique shortest coset representative, and the set of shortest coset representatives is denoted $\mathfrak{W}^J$. 
\end{definition}
Because of how frequently the set is referenced in later sections, we let 
\begin{equation}\label{eq:shortreps}
    \Wi \coloneqq \mathfrak{W}^{\{s_1,\ldots ,s_{n}\}\setminus\{s_i\}}
\end{equation}
denote the shortest coset representatives of $\mathfrak{W}_{\{s_1,\ldots,s_n\}\setminus \{s_i\}} = \mathfrak{S}_i \times \mathfrak{W}_{n-i}$.
\subsection{Root Systems}\label{ssec:rootsystems}
Denote by $\Phi$ the set of roots, $\Phi^+$, $\Phi^-$, and $\Delta$ as the set of positive, negative, and simple roots respectively. The set of positive roots is partially ordered by $\alpha < \beta$ if $\beta - \alpha$ is a nonnegative sum of elements in $\Delta$. Given a subset $H \subseteq \Phi^+$, let $S(H) \coloneqq \{s_\alpha \mid \alpha \in H\}$ be the set of reflections associated to the roots in $H$, so that, for example, $S(\Delta)$ is the set of simple reflections. \par
 The positive and simple roots of the type B root systems are 
\begin{equation}\label{eq:Broots}
    \Phi^+ \coloneqq \{e_i \pm e_{j} \mid 1 \leq i < j \leq n\} \cup \{e_i \mid i \in [n]\},
\end{equation}
 and 
\[\Delta \coloneqq \{e_i - e_{i+1} \mid i \in [n-1]\} \cup \{e_n\},\]
respectively. Usually one denotes $\alpha_i \coloneqq e_i-e_{i+1}$ and $\alpha_n \coloneqq e_n$, and a positive root $\alpha \in \Phi^+$ such that $\alpha  = \sum_{i=1}^n c_i\alpha_i$ as the vector $[c_1c_2\ldots c_n]$.\par 
The action of the reflection $s_n$ corresponding to $e_n$ on a vector is negation of the last coordinate in a vector. All other simple reflections $s_i$ act on a vector $v = \sum c_je_j$ by swapping the $i$-th index $c_i$ and $i+1$-th index $c_{i+1}$. \par 
    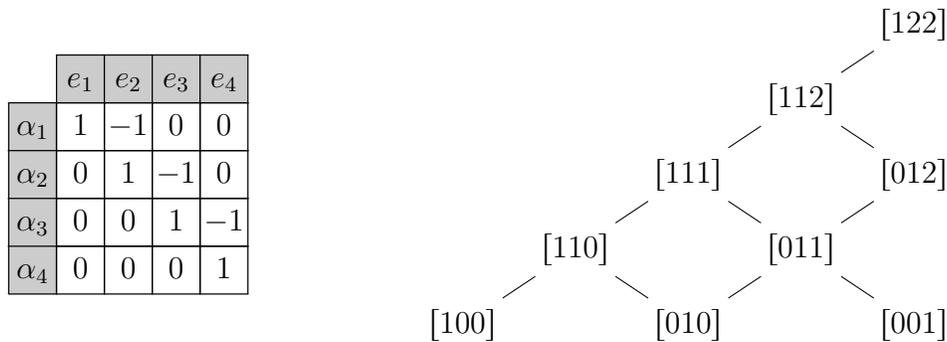
\begin{figure}[ht]
        \centering
    \begin{tikzpicture}
    \begin{scope}
        \draw (0,2) node (rts) {$\begin{ytableau}
        \none & *(light-gray) e_1 &*(light-gray)e_2 & *(light-gray)e_3 & *(light-gray)e_4 \\
        *(light-gray) \alpha_1 &1&-1&0&0 \\
        *(light-gray) \alpha_2 &0&1&-1&0 \\
        *(light-gray) \alpha_3 &0&0&1&-1 \\
        *(light-gray) \alpha_4 &0&0&0&1 
    \end{ytableau}$};
    \end{scope}
    \begin{scope}[xshift=1.75in]
        \draw (0,0) node (a1) {$[100]$};
        \draw (3,0) node (a2) {$[010]$};
        \draw (6,0) node (a3) {$[001]$};
        \draw (1.5,1) node (a12) {$[110]$};
        \draw (4.5,1) node (a23) {$[011]$};
        \draw (3,2) node (a123) {$[111]$};
        \draw (6,2) node (bc1) {$[012]$};
        \draw (4.5,3) node (bc2) {$[112]$};
        \draw (6,4) node (bc3) {$[122]$};

        \draw (a1)--(a12)--(a123);
        \draw (a2)--(a23)--(a123);
        \draw (a2)--(a12);
        \draw (a3)--(a23);
        \draw (a23)--(bc1);
        \draw (bc1)--(bc2)--(a123);
        \draw (bc2)--(bc3);
    \end{scope}
    \end{tikzpicture}
        \caption{The simple roots for B${}_4$ and positive root poset for B${}_3$}
        \label{fig:B4_roots}
    \end{figure}

These positive roots are in bijective correspondence with reflections in the type B Weyl group (i.e. $\Wn$). These reflections are precisely the transpositions in $\Wn$, given by how they act on the standard basis. Specifically, $S(\Phi^+)$ is the set of transpositions in $\Wn$, and in particular, 
\begin{equation}\label{eq:rot_trans_B}
    e_i-e_j \longleftrightarrow (i,j) \hspace{1cm} e_i+e_j \longleftrightarrow (i,\bar{j})  \hspace{1cm} e_i \longleftrightarrow (i,\bar{i}).
\end{equation}
\begin{example}\label{ex:B3_correspondence}
    In type B${}_3$, the correspondence between roots (see Figure \ref{fig:B4_roots}) and reflections in $\Wn$ are as follows.
    \begin{center} 
\begin{tabular}{|c|c|c|c|}\hline
    Root in $\Phi^+$  & Vector & Transposition in $\Wn$ & A Reduced Word \\\hline
    $[100]$ & $(1,-1,0)$ & $(1,2)$ & $s_1$ \\
    $[010]$ & $(0,1,-1)$ & $(2,3)$ & $s_2$\\
    $[001]$ & $(0,0,1)$  & $(3,\overline{3})$  & $s_3$\\
    $[110]$ & $(1,0,-1)$ & $(1,3)$ & $s_1s_2s_1$\\
    $[011]$ & $(0,1,0)$  & $(2,\overline{2})$  & $s_2s_3s_2$\\
    $[012]$ & $(0,1,1)$  & $(2,\overline{3})$  & $s_3s_2s_3$\\
    $[111]$ & $(1,0,0)$  & $(1,\overline{1})$  & $s_1s_2s_3s_2s_1$\\
    $[112]$ & $(1,0,1)$  & $(1,\overline{3})$  & $s_3s_1s_2s_1s_3$\\
    $[122]$ & $(1,1,0)$  & $(1,\overline{2})$  & $s_2s_3s_2s_1s_2s_3s_2$ \\\hline
\end{tabular}
\end{center}
Where, for example, the reflection across the vector $\alpha_1 + 2\alpha_2 + 2\alpha_3 = (1,1,0)$ sends $e_1 = (1,0,0)$ to $-e_2 = (0,-1,0)$, corresponding precisely to the transposition that switches $1$ and $\bar{2}$.
\end{example}
The type C root system is quite similar. The positive and simple roots of the type C root systems are 
\begin{equation}\label{eq:Croots}
    \Phi^+ \coloneqq \{e_i \pm e_{j} \mid 1 \leq i < j \leq n\} \cup \{2e_i \mid i \in [n]\},
\end{equation}
 and 
\[\Delta \coloneqq \{e_i - e_{i+1} \mid i \in [n-1]\} \cup \{2e_n\},\]
respectively. Usually one denotes $\alpha_i \coloneqq e_i-e_{i+1}$ and $\alpha_n \coloneqq 2e_n$, and a positive root $\alpha \in \Phi^+$ such that $\alpha  = \sum_{i=1}^n c_i\alpha_i$ as the vector $[c_1c_2\ldots c_n]$.\par
This may seem to be an insignificant change from type B, but there are several consequences. For example, in type B the action is $s_n(\alpha_{n-1}) = \alpha_{n-1} + 2\alpha_n$ and $s_{n-1}(\alpha_n) = \alpha_{n-1} + \alpha_n$, whereas in type C the action is $s_n(\alpha_{n-1}) = \alpha_{n-1}+\alpha_n$ and $s_{n-1}(\alpha_n) = 2\alpha_{n-1} + \alpha_n$.\par 

\begin{figure}[ht]
    \centering
        \begin{tikzpicture}
    \begin{scope}
        \draw (0,2) node (rts) {$ \begin{ytableau}
        \none & *(light-gray) e_1 &*(light-gray)e_2 & *(light-gray)e_3 & *(light-gray)e_4 \\
        *(light-gray) \alpha_1 &1&-1&0&0 \\
        *(light-gray) \alpha_2 &0&1&-1&0 \\
        *(light-gray) \alpha_3 &0&0&1&-1 \\
        *(light-gray) \alpha_4 &0&0&0&2 
    \end{ytableau}$};
    \end{scope}
    \begin{scope}[xshift=1.75in]
        \draw (0,0) node (a1) {$[100]$};
        \draw (3,0) node (a2) {$[010]$};
        \draw (6,0) node (a3) {$[001]$};
        \draw (1.5,1) node (a12) {$[110]$};
        \draw (4.5,1) node (a23) {$[011]$};
        \draw (3,2) node (a123) {$[111]$};
        \draw (6,2) node (bc1) {$[021]$};
        \draw (4.5,3) node (bc2) {$[121]$};
        \draw (6,4) node (bc3) {$[221]$};

        \draw (a1)--(a12)--(a123);
        \draw (a2)--(a23)--(a123);
        \draw (a2)--(a12);
        \draw (a3)--(a23);
        \draw (a23)--(bc1);
        \draw (bc1)--(bc2)--(a123);
        \draw (bc2)--(bc3);
    \end{scope}
    \end{tikzpicture}
    \caption{The simple roots for C${}_4$ and the positive root poset for C${}_3$.}
    \label{fig:C4_roots}
\end{figure}
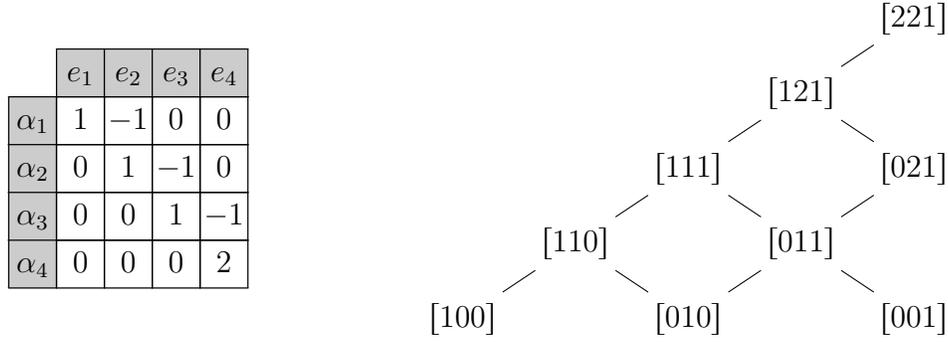

Reflections and transpositions for types B and C are identical, but the relations among them in the root poset is not. 
\begin{example}\label{ex:C3_correspondence}
    In type C${}_3$, the correspondence between roots (see Figure \ref{fig:C4_roots}) and reflections in $\Wn$ are as follows.
\begin{center} 
\begin{tabular}{|c|c|c|c|}\hline
    Root in $\Phi^+$ & Vector  & Transposition in $\Wn$ & Reflection $s_\alpha$\\\hline
    $[100]$ & $(1,-1,0)$ & $(1,2)$ & $s_1$\\
    $[010]$ & $(0,1,-1)$ & $(2,3)$ & $s_2$\\
    $[001]$ & $(0,0,2)$  & $(3,\overline{3})$  & $s_3$\\
    $[110]$ & $(1,0,-1)$ & $(1,3)$& $s_1s_2s_1$\\
    $[011]$ & $(0,1,1)$  & $(2,\overline{3})$ & $s_3s_2s_3$\\
    $[021]$ & $(0,2,0)$  & $(2,\overline{2})$ & $s_2s_3s_2$\\
    $[111]$ & $(1,0,1)$  & $(1,\overline{3})$ & $s_3s_1s_2s_1s_3$\\
    $[121]$ & $(1,1,0)$  & $(1,\overline{2})$ & $s_2s_3s_2s_1s_2s_3s_2$\\
    $[221]$ & $(2,0,0)$  & $(1,\overline{1})$ & $s_1s_2s_3s_2s_1$\\\hline
\end{tabular}
\end{center}
Where, for example, the reflection across the vector $\alpha_1 + 2\alpha_2 + \alpha_3 = (1,1,0)$ sends $e_1 = (1,0,0)$ to $-e_2 = (0,-1,0)$, corresponding precisely to the transposition that switches $1$ and $\bar{2}$.
\end{example}
There are two important takeaway from comparing Examples \ref{ex:B3_correspondence} and \ref{ex:C3_correspondence}. First is that the type B and C root posets (See Figures \ref{fig:B4_roots} and \ref{fig:C4_roots}) do \emph{not} induce the same partial order on reflections. Second, the order on transpositions induced by the root poset is \emph{not} Bruhat order. \par 
Some structure is the same. In either case the action of $\Wn$ as a real reflection group is identical. Moreover, since every element of $\Wn$ acts linearly, there is a natural action of $\Wn$ on the polynomial ring $\poly$ that is also the same in either type. 
\begin{equation}\label{eqn:wn_action_poly}
    w f\left(x_1,\ldots,x_n\right) \mapsto f\left(x_{w(1)},\ldots,x_{w(n)}\right).
\end{equation}
\begin{example}
    The following shows how $\mathfrak{W}_2$ act on some polynomials in $\C[x_1,x_2]$:
    \begin{center}
    \bgroup
    \setlength{\tabcolsep}{12pt} 
    \def\arraystretch{1.6}
    \begin{tabular}{c|ccccc}
                  & $x_1$ & $x_2$ & $x_1+x_2$ &  $x_1x_2$    & $x_1^2 + x_2^3$ \\\hline
        $[1,2]$             & $x_1$ & $x_2$ & $x_1+x_2$ &  $x_1x_2$    & $x_1^2 + x_2^3$ \\
        $[2,1]$             & $x_2$ & $x_1$ & $x_1+x_2$ &  $x_1x_2$    & $x_2^2 + x_1^3$ \\
        $[1,\bar{2}]$       & $x_1$ & $-x_2$ & $x_1-x_2$ & $-x_1x_2$   & $x_1^2 - x_2^3$ \\
        $[2,\bar{1}]$       & $x_2$ & $-x_1$ & $-x_1+x_2$  & $-x_1x_2$ & $x_2^2 - x_1^3$ \\
        $[\bar{2},1]$       & $-x_2$ & $x_1$ & $x_1-x_2$ & $-x_1x_2$   & $x_2^2 + x_1^3$ \\
        $[\bar{1},2]$       & $-x_1$ & $x_2$ & $-x_1+x_2$  & $-x_1x_2$ & $x_1^2 + x_2^3$ \\
        $[\bar{2},\bar{1}]$ & $-x_2$ & $-x_1$ & $-x_1-x_2$ & $x_1x_2$  & $x_2^2 - x_1^3$ \\
        $[\bar{1},\bar{2}]$ & $-x_1$ & $-x_2$ & $-x_1-x_2$ & $x_1x_2$  & $x_1^2 - x_2^3$ \\
    \end{tabular}
    \egroup
    \end{center}
\end{example}
 \subsection{Hessenberg spaces and inversions}
 A \emph{Hessenberg space} $H$ is a lower order ideal in the  positive root poset of a Weyl group\footnote{Those familiar with the geometry can rectify this description with the usual one by taking $-H \cup \Phi^+$.}. We insist $H$ contains the simple roots $\Delta$. 
 \begin{example}\label{ex:hessenbergs}
     Other than $H=\Delta$ there are $9$ Hessenberg spaces in C${}_3$:
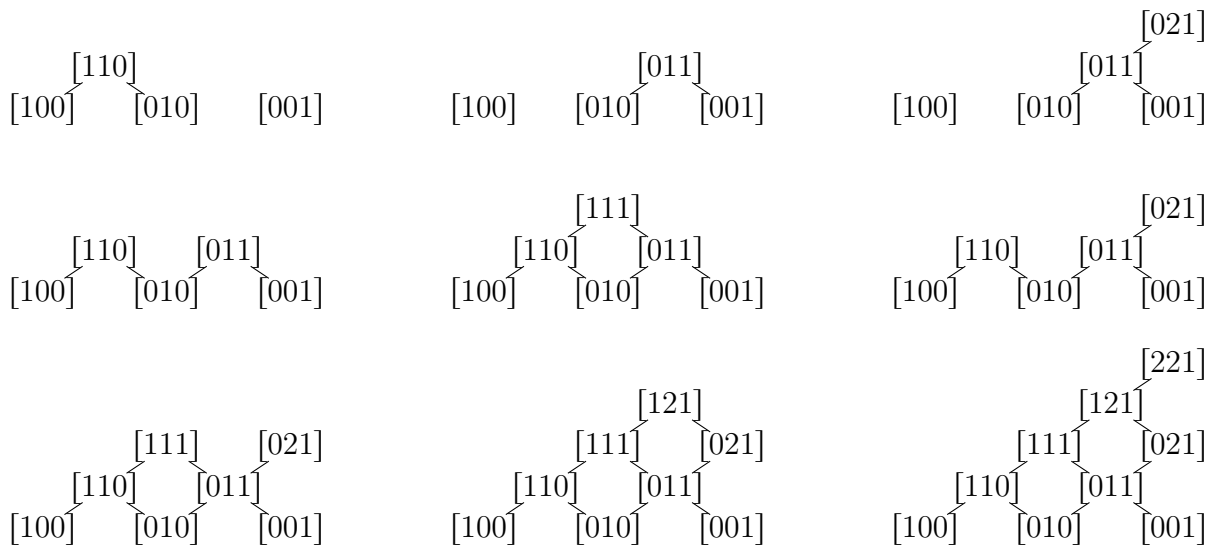
\begin{center}
    \begin{tikzpicture}[scale=0.55]
    \begin{scope}[xshift=0in]
        \draw (0,0) node (a1) {$[100]$};
        \draw (3,0) node (a2) {$[010]$};
        \draw (6,0) node (a3) {$[001]$};
        \draw (1.5,1) node (a12) {$[110]$};

        \draw (a1)--(a12);
        \draw (a2)--(a12);
    \end{scope}
    \begin{scope}[xshift=4.2in]
        \draw (0,0) node (a1) {$[100]$};
        \draw (3,0) node (a2) {$[010]$};
        \draw (6,0) node (a3) {$[001]$};
        \draw (4.5,1) node (a23) {$[011]$};

        \draw (a2)--(a23);
        \draw (a3)--(a23);
    \end{scope}
         \begin{scope}[xshift=8.4in]
        \draw (0,0) node (a1) {$[100]$};
        \draw (3,0) node (a2) {$[010]$};
        \draw (6,0) node (a3) {$[001]$};
        \draw (4.5,1) node (a23) {$[011]$};
        \draw (6,2) node (bc1) {$[021]$};

        \draw (a2)--(a23);
        \draw (a3)--(a23);
        \draw (a23)--(bc1);
    \end{scope}
    \begin{scope}[xshift=0in,yshift=-1.75in]
        \draw (0,0) node (a1) {$[100]$};
        \draw (3,0) node (a2) {$[010]$};
        \draw (6,0) node (a3) {$[001]$};
        \draw (1.5,1) node (a12) {$[110]$};
        \draw (4.5,1) node (a23) {$[011]$};

        \draw (a1)--(a12);
        \draw (a2)--(a23);
        \draw (a2)--(a12);
        \draw (a3)--(a23);
    \end{scope}
     \begin{scope}[xshift=4.2in,yshift=-1.75in]
        \draw (0,0) node (a1) {$[100]$};
        \draw (3,0) node (a2) {$[010]$};
        \draw (6,0) node (a3) {$[001]$};
        \draw (1.5,1) node (a12) {$[110]$};
        \draw (4.5,1) node (a23) {$[011]$};
        \draw (3,2) node (a123) {$[111]$};

        \draw (a1)--(a12)--(a123);
        \draw (a2)--(a23)--(a123);
        \draw (a2)--(a12);
        \draw (a3)--(a23);
    \end{scope}
     \begin{scope}[xshift=8.4in,yshift=-1.75in]
        \draw (0,0) node (a1) {$[100]$};
        \draw (3,0) node (a2) {$[010]$};
        \draw (6,0) node (a3) {$[001]$};
        \draw (1.5,1) node (a12) {$[110]$};
        \draw (4.5,1) node (a23) {$[011]$};
        \draw (6,2) node (bc1) {$[021]$};

        \draw (a1)--(a12);
        \draw (a2)--(a23);
        \draw (a2)--(a12);
        \draw (a3)--(a23);
        \draw (a23)--(bc1);
    \end{scope}
    \begin{scope}[xshift=0in,yshift=-4in]
        \draw (0,0) node (a1) {$[100]$};
        \draw (3,0) node (a2) {$[010]$};
        \draw (6,0) node (a3) {$[001]$};
        \draw (1.5,1) node (a12) {$[110]$};
        \draw (4.5,1) node (a23) {$[011]$};
        \draw (3,2) node (a123) {$[111]$};
        \draw (6,2) node (bc1) {$[021]$};

        \draw (a1)--(a12)--(a123);
        \draw (a2)--(a23)--(a123);
        \draw (a2)--(a12);
        \draw (a3)--(a23);
        \draw (a23)--(bc1);
    \end{scope}
        \begin{scope}[xshift=4.2in,yshift=-4in]
        \draw (0,0) node (a1) {$[100]$};
        \draw (3,0) node (a2) {$[010]$};
        \draw (6,0) node (a3) {$[001]$};
        \draw (1.5,1) node (a12) {$[110]$};
        \draw (4.5,1) node (a23) {$[011]$};
        \draw (3,2) node (a123) {$[111]$};
        \draw (6,2) node (bc1) {$[021]$};
        \draw (4.5,3) node (bc2) {$[121]$};

        \draw (a1)--(a12)--(a123);
        \draw (a2)--(a23)--(a123);
        \draw (a2)--(a12);
        \draw (a3)--(a23);
        \draw (a23)--(bc1);
        \draw (bc1)--(bc2)--(a123);
    \end{scope}
        \begin{scope}[xshift=8.4in,yshift=-4in]
        \draw (0,0) node (a1) {$[100]$};
        \draw (3,0) node (a2) {$[010]$};
        \draw (6,0) node (a3) {$[001]$};
        \draw (1.5,1) node (a12) {$[110]$};
        \draw (4.5,1) node (a23) {$[011]$};
        \draw (3,2) node (a123) {$[111]$};
        \draw (6,2) node (bc1) {$[021]$};
        \draw (4.5,3) node (bc2) {$[121]$};
        \draw (6,4) node (bc3) {$[221]$};

        \draw (a1)--(a12)--(a123);
        \draw (a2)--(a23)--(a123);
        \draw (a2)--(a12);
        \draw (a3)--(a23);
        \draw (a23)--(bc1);
        \draw (bc1)--(bc2)--(a123);
        \draw (bc2)--(bc3);
    \end{scope}
    \end{tikzpicture}
\end{center}
 \end{example}
 \begin{definition}\label{hinv}
     For a given Hessenberg space $H$ and $w \in \Wn$, an $H$-inversion of $w$ is, equivalently:
     \begin{itemize}
         \item a root $\alpha$ in $H$ that $w$ sends to a negative root (i.e. $w(\alpha) \in \Phi^-$), or
         \item a reflection $t \in S(H)$ where $wt < w$ in Bruhat order (i.e. $\ell(wt) < \ell(w)$).
     \end{itemize}
 \end{definition}
 \begin{remark}\label{rem:hinvs_and_invs}
     If $H = \Phi^+$, then the $H$-inversions are simply the (right) inversions of $w$. On the other hand, if $H = \Delta$ then these are the (right) descents of $w$.
 \end{remark}

 \subsection{Splines and the Dot Action}\label{ssec:splines_back}
This subsection introduces the core object of study: splines on the group of signed permutations. It also explains some of the geometric objects that motivate this research, and what properties of splines one can deduce from this connection.
\begin{wrapfigure}{r}{0.5\textwidth}\vspace{-0.7cm}
        \centering
        \begin{tikzpicture}[scale=1.9]
            \draw (0,0.2) node (e) {$[1,2]$};
            \draw (-1,1) node (l1) {$[2,1]$};
            \draw (-1,2) node (l2) {$[2,\bar{1}]$};
            \draw (-1,3) node (l3) {$[\bar{1},2]$};
            \draw (1,1) node (r1) {$[1,\bar{2}]$};
            \draw (1,2) node (r2) {$[\bar{2},1]$};
            \draw (1,3) node (r3) {$[\bar{2},\bar{1}]$};
            \draw (0,3.8) node (w0) {$[\bar{1},\bar{2}]$};

            \draw (e)--(l1) node[midway,left] {\textcolor{red}{$x_1-x_2$}};
            \draw (l1)--(l2) node[midway,left] {\textcolor{red}{$x_1\;$}};
            \draw (l2)--(l3) node[midway,left] {\textcolor{red}{$x_1+x_2$}};
            \draw (l3)--(w0) node[midway,left] {\textcolor{red}{$x_2\;$}};
            \draw (e)--(r1) node[midway,right] {\textcolor{red}{$\;x_2$}};
            \draw (r2)--(r3) node[midway,right] {\textcolor{red}{$x_1$}};
            \draw (r1)--(r2) node[midway,right] {\textcolor{red}{$x_1+x_2$}};
            \draw (r3)--(w0) node[midway,right] {\textcolor{red}{$\;x_1-x_2$}};
        \end{tikzpicture}%\vspace{-3pt}
        \caption{The labeled graph for $H = \Delta$.}
        \label{fig:W2_labels}\vspace{-1.2cm}
\end{wrapfigure}
\begin{definition}
Let $G = (V,E)$ be a simple graph and $\cl \coloneqq E \to \poly$ be a labeling of its edges by polynomials. A \emph{spline} on $G$ is an assignment $\spline{\rho} \colon V \to \poly$ such that if $(v_1,v_2) \in E$ then $\spline{\rho}(v_1)-\spline{\rho}(v_2)$ is divisible by $\cl(v_1,v_2)$.
\end{definition}

The splines in this paper are on graphs whose vertex set is $\Wn$ and whose edge set is $\{(w,ws) \mid w \in \Wn,\; s \in S(H)\}$ for a given Hessenberg space $H$. The labeling function $\cl(w,ws)$ is as follows. For every right multiplication of $w$ by a transposition $s_\alpha$, there is a corresponding left multiplication by a transposition $s_\beta$ so that $s_\beta w = ws_\alpha$. Recall that right-multiplication by a transposition switches and/or negates \emph{positions} in the signed permutation. The corresponding left multiplication is the transposition that swaps the \emph{values} (see Figure \ref{fig:W2_labels}). Then 
\begin{equation}\label{eq:labels}
\cl(w,ws) = \begin{cases}
    x_i - x_j = x_{w(p)} - x_{w(q)} &\text{if }(i,j)w = ws = w(p,q) \\
    x_i + x_j = x_{w(p)} - x_{w(q)} &\text{if }(i,\bj)w = ws = w(p,q) \\
    x_i = \abs{x_{w(j)}} &\text{if }(i,\bi)w = ws = w(j,\bj) \\
\end{cases}
\end{equation}
\begin{remark}\label{rem:Lie_labeling}
     In arbitrary Lie type, the label function $\cl(w,ws_\alpha)$ for $\alpha \in \Phi^+$ is the root $w(\alpha)$, treated as a polynomial in the usual sense (i.e. $e_i \mapsto x_i$). It is not difficult to show that this is the same labeling convention as above using the standard $\alpha_i = e_i-e_j$ and $\alpha_n = e_n$ or $2e_{n}$. Since we are concerned primarily with containment within the ideal $\llangle \cl(w,ws)\rrangle$ (and are working over a field), the ideals that these evaluations define are the same as \eqref{eq:labels} for both types B and C. \par
\end{remark}

\begin{definition}\label{def:splines}
    Let $H$ be a Hessenberg space for $\Wn$. The (graded) \emph{ring of splines} for $H$ is 
    \[ \splines{H} = \bigoplus_{i \geq 0} \splines{H}^i \coloneqq \left\{\left. \spline{\rho} \in \prod_{w \in \Wn} \poly \right| \spline{\rho}(w) - \spline{\rho}(ws) \in \llangle \cl(w,ws) \rrangle\text{ when } s \in S(H)     \right\},\]
where addition and multiplication are pointwise, and the graded piece $\splines{H}^k$ is the $\C$-vector space of splines that are homogeneous polynomials of degree $k$ on their support.
\end{definition}
\begin{remark}
    The labels used to define $\splines{H}$ in the introduction give precisely the same conditions as those here, since $\llangle x_{w(i)} - x_{w(\bar{i})} \rrangle = \llangle 2x_{w(i)} \rrangle = \llangle \abs{x_{w(i)}} \rrangle $.
\end{remark}
\begin{example}\label{ex:splines}
    The following are two splines on $\mathfrak{W}_2$, one in $\splines{\Delta}$ and one in $\splines{\Phi^+}$ (which have the same $S(H)$ in either types).
    
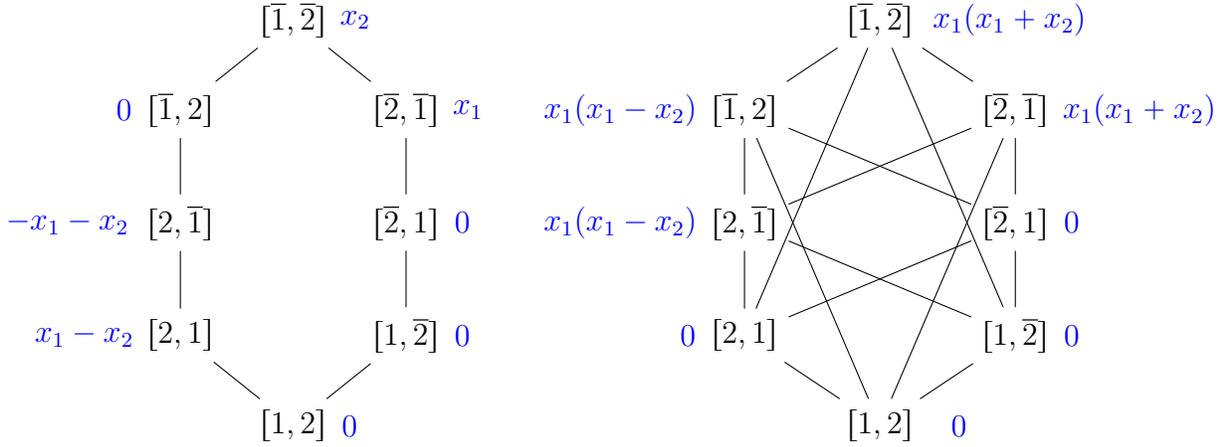
\begin{center}\hspace{-16pt}
    \begin{tikzpicture}[scale=1.5]
    \begin{scope}
        \draw (0,0.2) node (e) {$[1,2]$};
        \draw (-1,1) node (l1) {$[2,1]$};
        \draw (-1,2) node (l2) {$[2,\bar{1}]$};
        \draw (-1,3) node (l3) {$[\bar{1},2]$};
        \draw (1,1) node (r1) {$[1,\bar{2}]$};
        \draw (1,2) node (r2) {$[\bar{2},1]$};
        \draw (1,3) node (r3) {$[\bar{2},\bar{1}]$};
        \draw (0,3.8) node (w0) {$[\bar{1},\bar{2}]$};

        \draw  (0.5,0.2) node (se) {\textcolor{blue}{$0$}};
        \draw (-1.85,1) node (sl1)  {\textcolor{blue}{$x_1-x_2$}};
        \draw (-2.0,2) node (sl2)  {\textcolor{blue}{$-x_1-x_2$}};
        \draw (-1.5,3) node (sl3)  {\textcolor{blue}{$0$}};
        
        \draw  (1.5,1) node (sr1)   {\textcolor{blue}{$0$}};
        \draw  (1.5,2) node (sr2)   {\textcolor{blue}{$0$}};
        \draw  (1.55,3) node (sr3)   {\textcolor{blue}{$x_1$}};
        
        \draw  (0.55,3.8) node (sw0) {\textcolor{blue}{$x_2$}};

        \draw (e)--(l1)--(l2)--(l3)--(w0);
        \draw (e)--(r1)--(r2)--(r3)--(w0);
    \end{scope}
    \begin{scope}[xshift=5.2cm]
        \draw (0,0.2) node (e) {$[1,2]$};
        \draw (-1.2,1) node (l1) {$[2,1]$};
        \draw (-1.2,2) node (l2) {$[2,\bar{1}]$};
        \draw (-1.2,3) node (l3) {$[\bar{1},2]$};
        \draw (1.2,1) node (r1) {$[1,\bar{2}]$};
        \draw (1.2,2) node (r2) {$[\bar{2},1]$};
        \draw (1.2,3) node (r3) {$[\bar{2},\bar{1}]$};
        \draw (0,3.8) node (w0) {$[\bar{1},\bar{2}]$};
        
        \draw  (0.7,0.2) node (se) {\textcolor{blue}{$0$}};
        \draw (-1.7,1) node (sl1)  {\textcolor{blue}{$0$}};
        \draw (-2.3,2) node (sl2)  {\textcolor{blue}{$x_1(x_1-x_2)$}};
        \draw (-2.3,3) node (sl3)  {\textcolor{blue}{$x_1(x_1-x_2)$}};
        \draw  (1.7,1) node (sr1)   {\textcolor{blue}{$0$}};
        \draw  (1.7,2) node (sr2)   {\textcolor{blue}{$0$}};
        \draw  (2.3,3) node (sr3)   {\textcolor{blue}{$x_1(x_1+x_2)$}};
        \draw  (1.15,3.8) node (sw0) {\textcolor{blue}{$x_1(x_1+x_2)$}};

        \draw (e)--(l1)--(l2)--(l3)--(w0);
        \draw (e)--(r1)--(r2)--(r3)--(w0);
        \draw (w0)--(l1)--(r2);
        \draw (l2)--(r3)--(e);
        \draw (w0)--(r1)--(l2);
        \draw (r2)--(l3)--(e);
    \end{scope}
    \end{tikzpicture}
    \end{center}
\end{example}
We now construct two sets of splines and the \emph{identity spline}, each of which are elements of $\splines{H}$ for all $H$. Let 
\begin{center} 
\begin{tabular}{lll}
    $\mathbb{1} \colon \Wn \to \poly$ &be $\mathbb{1}(w)\coloneqq 1$ & for all $w \in \Wn$,\\
    $\spline{t}_i \colon \Wn \to \poly$ &be $\spline{t}_i(w) \coloneqq x_i$ & for all $w\in \Wn$, $i\in [n]$, and\\
    $\spline{r}_i \colon \Wn \to \poly$ &be $\spline{r}_i(w) \coloneqq x_{w(i)}$ & for all $w\in \Wn$, $i\in [n]$.
\end{tabular}
\end{center}
We use $\mathbb{1}$ to denote both the identity spline and the trivial representation on $\Wn$, since as functions on $\Wn$ they are identical. It will always be clear from context which one is being referred to.
\begin{lemma}\label{lem:t_and_r}
    For any element $w \in \Wn$, transposition $s \in S(\Phi^+)$, and integer $i \in [n]$,
    \[\spline{t}_i(w) - \spline{t}_i(ws) \in \llangle \cl(w,ws)\rrangle\]
    and 
    \[\spline{r}_i(w) - \spline{r}_i(ws) \in \llangle \cl(w,ws)\rrangle.\]
\end{lemma}
\begin{proof}
    Since $\spline{t}_i$ takes the same value for every $w \in \Wn$, the difference is always zero and therefore in any ideal. \par 
    If $s(i) = i$, then $\spline{r}_i(w) - \spline{r}_i(ws) = 0$ as well. On the other hand if $s = (i,j)$ for any $j \in [\bar{n}]$, then 
    \[
    \spline{r}_i(w) - \spline{r}_i(ws) = x_{w(i)} - x_{w(j)} \in \llangle \cl(w,w(i,j))\rrangle.
    \]
\end{proof}
The ring $\splines{H}$ is an infinite-dimensional $\C$-vector space in the natural way, and can also be viewed as a finitely generated graded $\poly$-module in two ways via the following module actions:
\begin{equation}\label{eqn:left_action}
    f\left(x_1,\ldots,x_n\right)\!.\spline{\rho} = f\left(\spline{t}_1,\ldots,\spline{t}_n\right)\spline{\rho}
\end{equation}
and
\begin{equation}\label{eqn:right_action}
    f\left(x_1,\ldots,x_n\right)\!.\spline{\rho} = f\left(\spline{r}_1,\ldots,\spline{r}_n\right)\spline{\rho},
\end{equation}
where the right-hand side of both (\ref{eqn:left_action}) and (\ref{eqn:right_action}) work by substituting splines for variables in to the polynomial $f$ then multiplying as in the ring structure of $\splines{H}$. For both actions the constant $f(0,\ldots,0)$ is naturally mapped to $f(0,\ldots,0)\mathbb{1}$. Since $\splines{H}$ is a $\poly$-submodule of $\prod_{w \in \Wn} \poly$ for either module action, it is finitely generated. We call the module action (\ref{eqn:left_action}) the \emph{left action} and the module action (\ref{eqn:right_action}) the \emph{right action} of $\poly$ on $\splines{H}$.  Both the left and right actions are naturally compatible with the grading on $\splines{H}$, and $\splines{H}$ is a free module with respect to both actions (see Proposition \ref{prop:geom_properties}). \par 
\begin{example}\label{ex:poly_action}
    Let $\spline{\rho} \in \splines{H}$ and let $f(x_\bullet) = x_1^3 + x_2^2 + x_3$. \\
    The left action of $f$ on $\spline{\rho}$ evaluated at any $v \in \Wn$ is: \[ f(x_\bullet).\spline{\rho}(v) = \left[((\spline{t}_1)^3 + (\spline{t}_2)^2 + \spline{t}_3)\spline{\rho}\right](v) = (x_1^3 + x_2^2 + x_3)\spline{\rho}(v),\]
    the right action of $f$ on $\spline{\rho}$ evaluated at any $v \in \Wn$ is: \[ f(x_\bullet).\spline{\rho}(v) =\left[((\spline{r}_1)^3 + (\spline{r}_2)^2 + \spline{r}_3)\spline{\rho}\right](v) = (x_{v(1)}^3 + x_{v(2)}^2 + x_{v(3)})\spline{\rho}(v).\]
    If $v = (1,\bar{1})(2,3)$, then $(x_{v(1)}^3 + x_{v(2)}^2 + x_{v(3)}) = ((-x_{1})^3 + x_{3}^2 + x_{2})$.
\end{example}
The ring of splines has a $\Wn$-module structure \cite{tymoczko2008permutation,tymoczko2008schubertreps}. 
\begin{definition}
    Let $\spline{\rho} \in \splines{H}$. The \emph{dot action} of $\Wn$ on $\splines{H}$ is given by 
    \[
    w\cdot \spline{\rho}(v) \coloneqq w\spline{\rho}(w^{-1}v)
    \]
    for $w,v \in \Wn$.
\end{definition}
Using our standard for visualizing splines, the dot action by $w$ moves polynomials around the $\Wn$ by sending the polynomial at $v$ to $wv$ (for all $v \in \Wn$), then acts on every polynomial by $w$ as in Equation (\ref{eqn:wn_action_poly}). \begin{example}\label{ex:dot_action}
The dot action of the transposition $(1,\bar{2})$ on one of the splines from Example \ref{ex:splines} is 
\begin{center}
\begin{tikzpicture}[scale=1.5]
    \begin{scope}
        \draw (0,0.2) node (e) {$[1,2]$};
        \draw (-1,1) node (l1) {$[2,1]$};
        \draw (-1,2) node (l2) {$[2,\bar{1}]$};
        \draw (-1,3) node (l3) {$[\bar{1},2]$};
        \draw (1,1) node (r1) {$[1,\bar{2}]$};
        \draw (1,2) node (r2) {$[\bar{2},1]$};
        \draw (1,3) node (r3) {$[\bar{2},\bar{1}]$};
        \draw (0,3.8) node (w0) {$[\bar{1},\bar{2}]$};

        \draw  (0.5,0.2) node (se) {\textcolor{blue}{$0$}};
        \draw (-2.0,1) node (sl1)  {\textcolor{blue}{$x_1-x_2$}};
        \draw (-2.0,2) node (sl2)  {\textcolor{blue}{$-x_1-x_2$}};
        \draw (-1.5,3) node (sl3)  {\textcolor{blue}{$0$}};
        
        \draw  (1.5,1) node (sr1)   {\textcolor{blue}{$0$}};
        \draw  (1.5,2) node (sr2)   {\textcolor{blue}{$0$}};
        \draw  (1.55,3) node (sr3)   {\textcolor{blue}{$x_1$}};
        
        \draw  (0.55,3.8) node (sw0) {\textcolor{blue}{$x_2$}};

        \draw (e)--(l1)--(l2)--(l3)--(w0);
        \draw (e)--(r1)--(r2)--(r3)--(w0);

        \draw (-3.15,2) node (t) {$(1,\bar{2}) \;\;  \cdot$};
        \draw  (2.25,2) node (eq) {\Large $=$};
    \end{scope}
    \begin{scope}[xshift=5.1cm]
        \draw (0,0.2) node (e) {$[1,2]$};
        \draw (-1,1) node (l1) {$[2,1]$};
        \draw (-1,2) node (l2) {$[2,\bar{1}]$};
        \draw (-1,3) node (l3) {$[\bar{1},2]$};
        \draw (1,1) node (r1) {$[1,\bar{2}]$};
        \draw (1,2) node (r2) {$[\bar{2},1]$};
        \draw (1,3) node (r3) {$[\bar{2},\bar{1}]$};
        \draw (0,3.8) node (w0) {$[\bar{1},\bar{2}]$};

        \draw  (0.57,0.2) node (se) {\textcolor{blue}{$-x_2$}};
        \draw (-1.57,1) node (sl1)  {\textcolor{blue}{$-x_1$}};
        \draw (-1.5,2) node (sl2)  {\textcolor{blue}{$0$}};
        \draw (-2,3) node (sl3)  {\textcolor{blue}{$x_1+x_2$}};
        
        \draw  (1.5,1) node (sr1)   {\textcolor{blue}{$0$}};
        \draw  (1.5,2) node (sr2)   {\textcolor{blue}{$0$}};
        \draw  (1.5,3) node (sr3)   {\textcolor{blue}{$0$}};
        
        \draw  (1,3.8) node (sw0) {\textcolor{blue}{$-x_2+x_1$}};

        \draw (e)--(l1)--(l2)--(l3)--(w0);
        \draw (e)--(r1)--(r2)--(r3)--(w0);
    \end{scope}
    \end{tikzpicture}
\end{center}
\end{example}
\begin{remark} The dot action is well-defined. For all $(v_1,v_2) \in E(G)$,
\begin{align*}
w \cdot \spline{\rho}(v_1) - w \cdot \spline{\rho}(v_2) &= w\spline{\rho}(w^{-1}v_1)- w\spline{\rho}(w^{-1}v_2) \\
&= w(\spline{\rho}(w^{-1}v_1)- \spline{\rho}(w^{-1}v_2)) \\
&\in w\llangle \cl(w^{-1}v_1,w^{-1}v_2)\rrangle .
\end{align*}
If $v_1v_2^{-1} = (i,j)$ then $w^{-1}v_1v_2^{-1}w = (w^{-1}(i),w^{-1}(j))$. So 
\[w\cl(w^{-1}v_1,w^{-1}v_2) = \llangle w(t_{w^{-1}(i)}-t_{w^{-1}(j)}) \rrangle = \llangle t_i-t_j \rrangle  = \cl(v_1,v_2).\] 
Thus $w \cdot \spline{\rho}(v_1) - w \cdot \spline{\rho}(v_2) \in \llangle \cl(v_1,v_2)\rrangle$, and $w \cdot \spline{\rho} \in \splines{H}$.
\end{remark}
Finally, consider the quotients 
\begin{equation}\label{eqndef:left_quotient}
    \lqot{H} = \bigoplus_{i\geq 0} \left(\lqot{H}\right)_i \coloneqq \faktor{\splines{H}}{\llangle \spline{t}_1,\ldots,\spline{t}_n\rrangle}
\end{equation}
and
\begin{equation}\label{eqndef:right_quotient}
    \rqot{H} = \bigoplus_{i\geq 0} \left(\rqot{H}\right)_i \coloneqq \faktor{\splines{H}}{\llangle \spline{r}_1,\ldots,\spline{r}_n\rrangle}
\end{equation}
Call $\lqot{H}$ and $\rqot{H}$ the \emph{left} and \emph{right} quotients of $\splines{H}$ respectively. As $\poly$-modules, these quotients are $\faktor{\splines{H}}{I\splines{H}}$ where $I$ is the ``irrelevant ideal" $\llangle x_1,\ldots ,x_n\rrangle$ of $\polyr{n}$. Thus, $\lqot{H}$ and $\rqot{H}$ each inherit the structure of a finite-dimensional graded $\C$-vector space from the left- and right-module structures of $\splines{H}$ respectively. Any homogeneous module-generating set over $\poly$ projects to a spanning set over $\C$ in the quotient, and a minimal one projects to a basis (see Proposition \ref{prop:geom_properties}).\par 
The ideals $\llangle \spline{t}_1,\ldots,\spline{t}_n\rrangle$ and $\llangle \spline{r}_1,\ldots,\spline{r}_n\rrangle$ are homogeneous and $\Wn$-equivariant, and so the graded $\Wn$-module structure on $\splines{H}$ projects to graded $\Wn$-representations on both $\lqot{H}$ and $\rqot{H}$. Denote the character of these (graded) $\Wn$-representations as $\lrep{H}$ and $\rrep{H}$ respectively. \par 
These constructions are strongly motivated by geometry. This ring of splines is isomorphic to the equivariant cohomology ring of a corresponding regular semisimple Hessenberg variety $\Hess(X,H)$ via GKM Theory \cite{DMPS1992hessenbergvarieties}.
\begin{theorem}[\cite{DMPS1992hessenbergvarieties,GKM_theory}]
    Let $H$ be a Hessenberg space and $X$ a regular semisimple element of the Lie algebra. Then there exists a graded isomorphism
    \[ \Psi \colon H_T^*(\Hess(X,H) \to \splines{H} \]
    sending the $2i$-th equivariant cohomology $H_T^{2i}(\Hess(X,H))$ to the $i$-th graded piece $\splines{H}^i$.\par 
    Moreover, if $I=\llangle x_1,\ldots,x_n\rrangle$, then 
    \[H^*(\Hess(X,H)) \cong \faktor{\splines{H}}{I\splines{H}} = \lqot{H}\]
    as a graded vector space.
\end{theorem}
So the left quotient $\lqot{H}$ gives the ordinary cohomology of the regular semisimple Hessenberg variety. In type A, the right quotient $\rqot{H}$ is the map to the ordinary cohomology of a different algebraic variety whose $T$-equivariant cohomology happens to be isomorphic to that of $\Hess(X,H)$ \cite{Ayzenberg2018isospectral}. There are several properties $\splines{H}$ that one may discern from well-know geometric results.
\begin{proposition}[\cite{GKM_theory,DMPS1992hessenbergvarieties,PrecupBCPaving,tymoczko2008schubertreps}]\label{prop:geom_properties}
    Let $H$ be a Hessenberg space.
    \begin{enumerate}
        \item $\splines{H}$ is a graded free module over $\poly$.
        \item The top nonzero degree of $\lqot{H}$ and $\rqot{H}$ is $|H|$, and the rank-generating function is palindromic. 
        \item There exists a homogeneous basis $\{\spline{\sigma}_w^H \mid w \in \Wn\}$ of $\splines{H}$ such that:
        \begin{enumerate}
            \item (Minimality) The evaluation $\spline{\sigma}_w^H(w)$ is the product of the labels for each $H$-inversions of $w$.
            \item (Upper Triangularity) Each element $v$ in the support of $\spline{\sigma}_w^H$ is either equal to $w$ or $\ell(v) > \ell(w)$.
        \end{enumerate}
        \item The rank-generating function is $\ds \sum_{w \in \Wn} q^{\ell_H(w)}$ where $\ell_H(w) \coloneqq \abs{\{s \in S(H) \mid ws < w\}}$.
    \end{enumerate}
    Moreover, any set of splines indexed by $\Wn$ satisfying minimality (3a) and upper triangularity (3b) is a basis for $\splines{H}$.
\end{proposition}
\begin{proof}
    (1) follows from being a GKM variety and (2) comes from being smooth of dimension $|H|$. (3) comes from a paving of this variety by Hessenberg-Schubert cells. (4) identifies the homogeneous degrees of each $\spline{\sigma}_w^H$. Any set satisfying (3a) and (3b) is an ``upper triangular" basis and follows from basic facts in more general spline theory.
\end{proof}
\begin{remark}
    While necessary in general degree, Proposition \ref{prop:geom_properties} is not strictly necessary for computing the degree \emph{one} piece of $\splines{H}$. In fact, when restricting to the module generated by linear and constant splines, freeness and the properties that guarantee a basis are given even in the absence of the geometry  (c.f. Lemma 3.2 in \cite{Lesnev25}).
\end{remark}
 We will construct a set of splines that satisfies the minimality and upper-triangularity conditions in Proposition \ref{prop:geom_properties} (3a-b) in degree $0$ and $1$, thereby computing a basis for the $\C$-vector space $\splines{H}^1$. This is trivial for degree $0$, since $\Delta \subseteq H$ and $e$ is the only element in $\Wn$ with no descents. \par 
 So the plan of attack is clear. First, identify all elements $w \in \Wn$ where there is precisely one $s \in S(H)$ such that $ws < w$. Then, for each such $w$ define splines in $\splines{H}^1$ that are supported at $w$ and longer elements that evaluate to $\cl(w,ws)$ at $w$. Finally, reduce with linear relations until only $\dim(\splines{H}^1)$ remain, in such a way that computing the dot action remains feasible.
 \begin{definition}\label{def:hdes}
    For a Hessenberg space $H$ and integer $i \in [n]$, the (unique) \emph{$H$-descent set} of $i$ is
    \[
    \hdes{i} \coloneqq \left\{ w \in \Wn \mid w(\alpha_i) \in \Phi^-,\; w(\beta) \in \Phi^+\text{ for all }\beta \in H \setminus \{\alpha_i\}\right\},
    \]
    or equivalently,
    \[
    \hdes{i} \coloneqq \left\{ w \in \Wn \mid ws_i < w ,\; wt > w \text{ for all }t \in S(H) \setminus \{s_i\}\right\}.
    \]
    Less formally, $\hdes{i}$ is the set of elements with a single descent at $i$ and no other $H$-inversions.
\end{definition}
Proposition \ref{prop:dim1} below is a consequence of \cite{PrecupBCPaving} (i.e. Proposition \ref{prop:geom_properties}(4)) and a simple calculation for the image of $I\splines{H}^0$ in $\splines{H}^1$.
\begin{proposition}\label{prop:dim1}
    The dimension of $\splines{H}^i$ is $\ds n + \sum_{i=1}^n\hdes{i}$.
\end{proposition}
\begin{remark}
    We can already build splines with these properties for the identity element $e$ and for simple reflections $\{s_i \mid i \in [n]\}$. The spline $\mathbb{1}$ is supported at the identity $e$ and above,  $es \not< e$ for any $s \in \Wn$, and $\mathbb{1}$ evaluates to a unit at $e$. For any $i\in [n]$ the spline $\ds \sum_{j=1}^i \left(\spline{r}_j - \spline{t}_j\right)$ has $s_i$ as the shortest element in its support and the evaluation at $s_i$ is $\cl(e,s_i)$. To readers of \cite{ayzenberg2022second,Lesnev25}, it might be surprising to see that  $\ds \sum_{j=1}^n \left(\spline{r}_j - \spline{t}_j\right) \not\equiv 0$, but is in fact supported at $s_n$.
\end{remark}

\section{Ideals of Transpositions in Types B and C}\label{sec:ideals_of_transpositions}
Consider the poset of positive roots \eqref{eq:Broots} for type B and \eqref{eq:Croots} for type C. 
\begin{lemma}\label{lem:BCequality}
    Let $H$ be a type B Hessenberg space and $H'$ be a type C Hessenberg space. If $S(H) = S(H')$ then $\splines{H} = \splines{H'}$.
\end{lemma}
\begin{proof}
    From Definition \ref{def:splines}, this is immediate. 
    If one insists on using the general definition for $\cl(w,ws_\alpha)$ as a weight (see Remark \ref{rem:Lie_labeling}), then consider that $s_\alpha = s_\beta$ where $\alpha$ is a type B root and $\beta$ a type C root if and only if $\alpha$ and $\beta$ differ by scalar multiplication. So the edge-ideal is the same.
\end{proof}
\begin{remark}
    We believe that the reverse direction of Lemma \ref{lem:BCequality} is also true, that if $H$ and $H'$ are corresponding Hessenberg spaces in types B and C under the natural isomorphism, then $S(H) = S(H')$ if and only if $\splines{H} = \splines{H'}$. Even stronger, we ask the following: if $S(H) \neq S(H')$, is it possible for the rank-generating functions for $\splines{H}$ and $\splines{H'}$ to be equal?
\end{remark}
\begin{example}\label{ex:BCcorrespondence}
    Consider the positive root posets for B${}_4$ and C${}_4$. Replace each root with their corresponding transposition to get the following posets.
    
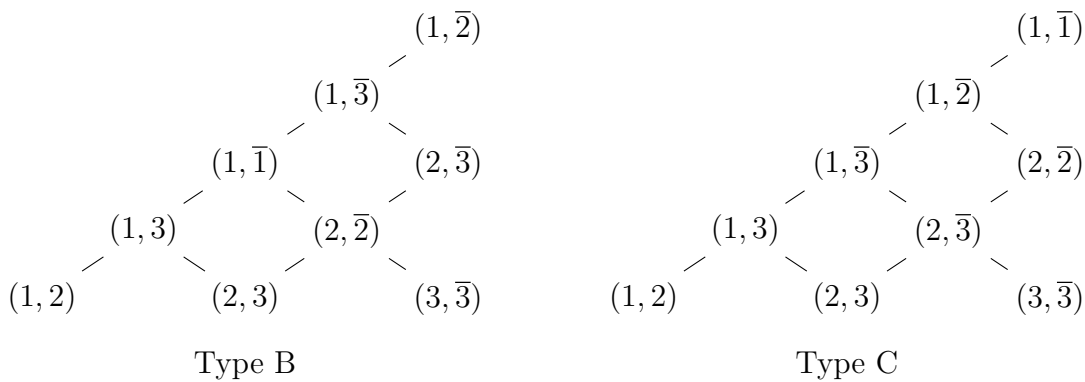
\begin{center}
    \begin{tikzpicture}[scale=0.9]
    \begin{scope}
\begin{scope}[xshift=0in]
        \draw (0,0) node (a1) {$(1,2)$};
        \draw (3,0) node (a2) {$(2,3)$};
        \draw (6,0) node (a3) {$(3,\bar{3})$};
        \draw (1.5,1) node (a12) {$(1,3)$};
        \draw (4.5,1) node (a23) {$(2,\bar{2})$};
        \draw (3,2) node (a123) {$(1,\bar{1})$};
        \draw (6,2) node (bc1) {$(2,\bar{3})$};
        \draw (4.5,3) node (bc2) {$(1,\bar{3})$};
        \draw (6,4) node (bc3) {$(1,\bar{2})$};

        \draw (a1)--(a12)--(a123);
        \draw (a2)--(a23)--(a123);
        \draw (a2)--(a12);
        \draw (a3)--(a23);
        \draw (a23)--(bc1);
        \draw (bc1)--(bc2)--(a123);
        \draw (bc2)--(bc3);

    \draw (3,-1) node (name) {Type B};
    \end{scope}
    \end{scope}
    \begin{scope}[xshift=3.5in]
        \draw (0,0) node (a1) {$(1,2)$};
        \draw (3,0) node (a2) {$(2,3)$};
        \draw (6,0) node (a3) {$(3,\bar{3})$};
        \draw (1.5,1) node (a12) {$(1,3)$};
        \draw (4.5,1) node (a23) {$(2,\bar{3})$};
        \draw (3,2) node (a123) {$(1,\bar{3})$};
        \draw (6,2) node (bc1) {$(2,\bar{2})$};
        \draw (4.5,3) node (bc2) {$(1,\bar{2})$};
        \draw (6,4) node (bc3) {$(1,\bar{1})$};

        \draw (a1)--(a12)--(a123);
        \draw (a2)--(a23)--(a123);
        \draw (a2)--(a12);
        \draw (a3)--(a23);
        \draw (a23)--(bc1);
        \draw (bc1)--(bc2)--(a123);
        \draw (bc2)--(bc3);

    \draw (3,-1) node (name) {Type C};
    \end{scope}
    \end{tikzpicture}
    \end{center}
    So any valid $S(H)$ for either type must correspond to a lower order ideal in either of the above posets. 
\end{example}
\begin{remark}\label{rem:not_ideals}
    It is tempting to define a unified B/C-poset on transpositions. In general, the union of sets of lower order ideals for two posets is \emph{not} the set of lower order ideals of another poset. This is already the case in $\mathfrak{W}_4$. One could certainly find a poset whose set of lower order ideals \emph{contain} all $S(H)$ in both types, but whether that construction would be combinatorially interesting (or geometrically meaningful for that matter) is unclear.
\end{remark}

\section{Simpler Conditions for \texorpdfstring{$\splines{H}^1$}{MH1}}\label{sec:trivial}
The goal of this section is to reduce the data one needs from any Hessenberg space in order to compute $\splines{H}^1$. In particular, this section reduces to the case that $H \subset \Delta \cup \{s_i(\alpha_j) \mid i \neq j\}$. \par  
Lemma \ref{lem:BCtriv} gives a condition for a root $\alpha \in \Phi^+$ to be added to $H$ without affecting $\splines{H}^1$.
\begin{lemma}\label{lem:BCtriv}
    If $H$ is a type B or type C Hessenberg space,  $\alpha \in \Phi^+$, and 
    \[
    \left\{ \beta \in \Delta \cup \{s_i(\alpha_j) \mid i \neq j\} \mid \beta < \alpha\right\} \subseteq H,  
    \]
    then $\splines{H}^1 = \splines{H \cup \alpha}^1$. 
\end{lemma}
Note that in type B,
\[
\Delta \cup \left\{s_i(\alpha_j) \mid i \neq j\right\} = \Delta \cup \{\alpha_{i} + \alpha_{i+1} \mid i \in [n-1]\} \cup \{\alpha_{n-1}+2\alpha_{n}\},
\]
and in type C, 
\[
\Delta \cup \left\{s_i(\alpha_j) \mid i \neq j\right\} = \Delta \cup \{\alpha_{i} + \alpha_{i+1} \mid i \in [n-1]\} \cup \{2\alpha_{n-1}+\alpha_{n}\}.
\]

\begin{proof}
    It follows from Definition \ref{def:splines} that $\splines{H} \supseteq \splines{H\cup\{\alpha\}}$. For any $\alpha$, it follows from  \ref{def:hdes} that $\hdes{i} \supseteq \cd_{H\cup \alpha}(i)$. We will prove that $\hdes{i} = \cd_{H\cup \alpha}(i)$, as if this holds then a basis for $\splines{H\cup\{\alpha\}}^1$ must also be a basis for $\splines{H}^1$.\par 
    The proof for types B and C are very similar. In either case, $\alpha$ can be written as a non-negative sum of roots in \[\left\{ \beta \in \Delta \cup \{s_i(\alpha_j) \mid i \neq j\} \mid \beta < \alpha\right\} \setminus \{\alpha_i\}.\]
    If $w \in \hdes{i}$ and $w(\alpha) \in \Phi^+$ then $w \in \cd_{H\cup\{\alpha\}}$. Also, if $w \in \hdes{i}$ then $w(\beta) \in \Phi^+$ for all $\beta \in H \setminus \{\alpha_i\}$, which contains the set above. It will follow that $w(\alpha) \in \Phi^+$ as $w$ is linear.\par 
    First, we verify that $\alpha$ can be written in this manned for type B in two cases. Let $w \in \hdes{i}$ be arbitrary, and note that $\alpha$ must be the sum of at least three distinct roots.\par 
    If $\alpha = \alpha_p + \cdots +\alpha_q$ for $q-p \geq 2$ then 
    \[
    \alpha = 
    \begin{cases}
        \alpha_p + \cdots + \alpha_q &\text{ for $i<p$ or $i>q$,}\\
        \alpha_p + \cdots + (\alpha_i + \alpha_{i+1}) + \cdots + \alpha_q &\text{ for $p \leq i < q$, and }\\
        \alpha_p + \cdots + (\alpha_{q-1} + \alpha_{q})&\text{ for $i = q$.}
    \end{cases}
    \]
    If $\alpha = \alpha_p + \cdots + \alpha_{q-1} + 2\alpha_q + \cdots + 2\alpha_n$ for $p < q \leq n$ and $n-p \geq 2$ then 
    \[
    \alpha = 
    \begin{cases}
        \alpha_p + \cdots + \alpha_{q-1} + 2\alpha_q + \cdots + 2\alpha_n &\text{ for $i<p$,}\\
        \alpha_p + \cdots + (\alpha_i + \alpha_{i+1}) + \cdots + \alpha_{q-1} + 2\alpha_q + \cdots + 2\alpha_n &\text{ for $p \leq i < q-1$,}\\
        \alpha_p + \cdots + (\alpha_{q-1} + \alpha_q) + \alpha_q + 2\alpha_{q+1} + \cdots + 2\alpha_n &\text{ for $p \leq i=q-1$,}\\
        \alpha_p + \cdots  + \alpha_{q-1} + 2\alpha_q + \cdots + 2(\alpha_i + \alpha_{i+1}) + \cdots + 2\alpha_n &\text{ for $q \leq i < n$, and }\\
        \alpha_p + \cdots + \alpha_{q-2} + \sum_{r=q}^{n-1}(\alpha_{r-1} + \alpha_r) + (\alpha_{n-1} +  2\alpha_n)&\text{ for $i = n$.}
    \end{cases}
    \]
    So $w(\alpha) \in \Phi^+$ and $w \in \cd_{H\cup\{\alpha\}}(i)$.\par 
    For type C the first case is identical and the second case \[\alpha = \alpha_p + \cdots + \alpha_{q-1} + 2\alpha_q + \cdots + 2\alpha_{n-1} + \alpha_n\]
    for $p \leq q < n$ and $n-p \leq 2$ is in fact easier, the only meaningful difference is that the case $i=n-1$, where 
    \[
    \alpha = \alpha_p + \cdots + \alpha_{q-1} + 2\alpha_q + \cdots + 2\alpha_{n-2} + (2\alpha_{n-1} +  \alpha_n).
    \]
    and in any case, $w(\alpha) \in \Phi^+$.\par 
    So for both types B and C it follows that $\hdes{i} = \cd_{H\cup\{\alpha\}}(i)$ for all $i \in [n]$, and so $\splines{H}^1 = \splines{H\cup\{\alpha\}}^1$.
\end{proof}
\begin{proposition}\label{prop:BCtriv}
    If $H$ is a type B or C Hessenberg space and \[H' \coloneqq H \cap \left(\Delta \cup \{s_i(\alpha_j) \mid i \neq j\}\right) \;\;\text{ then } 
    \;\;\splines{H}^1 = \splines{H'}^1.\]
\end{proposition}
\begin{proof}
    The set $H$ can be constructed from $H'$ by repeated application of Lemma \ref{lem:BCtriv}.
\end{proof}
So we may restrict to specifically $\Delta \subseteq H \subseteq \Delta \cup \{s_i(\alpha_j) \mid i \neq j\}$. In particular, $\splines{H}^1$ is determined entirely by the set
\[
S(H) \cap \left(\{s_is_{i+1}s_i \mid i \in [n]\} \cup \{s_ns_{n-1}s_n\}\right).
\]
Let  
\begin{equation}\label{eq:Tset}
\begin{aligned}
    t_i &\coloneqq s_is_{i+1}s_i = (i,i+2) &\longleftrightarrow \;\,\alpha_i + \alpha_{i+1}\;\;\text{ for }i\in [n-2]\hspace{0.15cm} \\
    t_{n-1} &\coloneqq s_{n-1}s_ns_{n-1} = (n-1,\overline{n-1}) &\longleftrightarrow\ \begin{cases} \alpha_{n-1}+\alpha_n & \text{ in type B} \\ 2\alpha_{n-1}+\alpha_n & \text{ in type C} \end{cases}\\
    t_n &\coloneqq s_ns_{n-1}s_n = (n-1,\bar{n})&\longleftrightarrow \begin{cases} \alpha_{n-1}+2\alpha_n & \text{ in type B} \\ \alpha_{n-1}+\alpha_n & \text{ in type C}. \end{cases}\\
    \end{aligned}
\end{equation}
\begin{corollary}\label{cor:t_transpositions}
        Let $t_i$ be as in \eqref{eq:Tset}. Then $\splines{H}^1$ is entirely determined by $\{t_i \mid i \in [n]\}\cap S(H)$, and every subset of $\{t_i \mid i \in [n]\}$ appears as $\{t_i \mid i \in [n]\}\cap S(H)$ for some type B or C Hessenberg space $H$.
\end{corollary}
\begin{remark}
    Another consequence of Proposition \ref{prop:BCtriv} is that if $\Delta \cup \{s_i(\alpha_j) \mid i \neq j\} \subseteq H$, then $\splines{H}^1 = \splines{\Phi^+}^1$. By GKM Theory, $\splines{\Phi^+}^1$ is the second equivariant cohomology of the full flag variety $G/B$. In other types, a similar calculation is possible. Specifically, $\splines{H}^1 = \splines{\Phi^+}^1$ whenever $H$ contains the following positive root ideals:
    \begin{enumerate}
        \item[A${}_n$:] $\Delta \sqcup \{\alpha_{i} + \alpha_{i+1} \mid i \in [n-1]\}$
        \item[B${}_n$:] $\Delta \sqcup \{\alpha_{i} + \alpha_{i+1} \mid i \in [n-1]\} \sqcup\{\alpha_{n-1} + 2\alpha_{n}\}$
        \item[C${}_n$:] $\Delta \sqcup \{\alpha_{i} + \alpha_{i+1} \mid i \in [n-1]\} \sqcup\{2\alpha_{n-1} + \alpha_{n}\}$
        \item[D${}_n$:] $\Delta \sqcup \{\alpha_{i} + \alpha_{i+1} \mid i \in [n-1]\} \sqcup\{2\alpha_{n-1} + \alpha_{n}\}$.
    \end{enumerate}
    In addition, for the exceptional types we may compute case-by-case. We display the order ideal generators for brevity.
    \begin{enumerate}
        \item[E${}_6$:] $\{\alpha_1+\alpha_3,\alpha_2+\alpha_4,\alpha_3+\alpha_4,\alpha_4+\alpha_5,\alpha_5+\alpha_6 \}$
        \item[E${}_7$:] $\{\alpha_1+\alpha_3,\alpha_2+\alpha_4,\alpha_3+\alpha_4,\alpha_4+\alpha_5,\alpha_5+\alpha_6,\alpha_6+\alpha_7 \}$
        \item[E${}_8$:] $\{\alpha_1+\alpha_3,\alpha_2+\alpha_4,\alpha_3+\alpha_4,\alpha_4+\alpha_5,\alpha_5+\alpha_6,\alpha_6+\alpha_7 ,\alpha_7+\alpha_8  \}$
        \item[F${}_4$:]$\{\alpha_1+\alpha_2,\alpha_2+2\alpha_3,\alpha_3+\alpha_4\}$
    \item[G${}_2$:] $\{3\alpha_1+\alpha_2\}$.
    \end{enumerate} 
\end{remark}

\section{Signed Permutations with a singe H-inversion}\label{sec:one_inversion}
This section explicitly determines the sets $\hdes{i}$ of elements in $\Wn$ whose unique $H$-inversions are the descent at $i$ from the set  $\{t_i \mid i \in [n]\}\cap S(H)$.
\begin{lemma}\label{lem:subword_endings}
    Say $w \in \Wn$ is length at least $2$ and $i \in [n]$. The following are equivalent:
    \begin{enumerate}
        \item Every reduced word for $w$ ends in $s_i$.
        \item $w \in \Wi \setminus \{e\}$.
        \item Every subword of $w$ ends in $s_{i+1}s_i$ or $s_{i-1}s_i$.
    \end{enumerate}
\end{lemma}
\begin{proof}
    Equivalence of (1) and (2) is the definition of a minimal coset representative. Since $s_is_j = s_js_i$ if and only if $\abs{i-j} \geq 2$, (3) is equivalent to (1).
\end{proof}
\begin{lemma}\label{lem:t_positivity_from_s}
For $i \in [n-1]$, if  $ws_i > w$ and $ws_{i+1} > w$ then $wt_i > w$. If $i={n-1}$, then  $wt_n > w$ as well. 
\end{lemma}
\begin{proof}
    In both types B and C the root associated to $t_i$ given in \eqref{eq:Tset} is a non-negative sum of $\alpha_i$ and $\alpha_{i+1}$. Since neither simple root is sent to $\Phi^-$ by $w$, neither is the root associated to $t_i$. 
\end{proof}
\begin{lemma}\label{lem:subword_partial_endings}
    Let $i \in [n-1]$. Every reduced word for a signed permutation $w$ ends in $s_{i+1}s_{i}$ if and only if \[
    w = s_j \cdots s_i \text{ for }j \geq i \;\; or \;\;w=s_j\cdots s_n \cdots s_i \text{ for } j \in [n-1].
    \]
    Every reduced word for $w$ ends in $s_{i-1}s_i$ if and only if 
    \[w = s_j \cdots s_i \text{ for } j < i.\]
\end{lemma}
\begin{proof}
    The words $s_j \cdots s_i$ and $s_j \cdots s_n \cdots s_i$ have no braid or commuting moves, and so these are the only reduced words for these elements of $\Wn$. This gives the backwards direction. \par 
    Assume all reduced words for $w$ end in $s_{i+1}s_i$ (resp. $s_{i-1}s_i$). Then every final subword of $w$ (that isn't simply $s_i$) does as well. \par 
    Proceed by induction on the length of $w$. If $w$ is length $2$, the claim is true and $w = s_{i+1}s_i$ (resp. $w=s_{i-1}s_i$). Say that $w = s_k\sigma > \sigma$, that every reduced word for $w$ ends in $s_{i+1}s_i$ (resp. $s_{i-1}s_i$), and so by induction $\sigma$ takes one of the forms above. \par
    Now consider $s_k$ so that $w$ is not of the form above.
    \begin{itemize}
        \item If $\sigma = s_j \cdots s_i$ for $j < i$ and $k \neq j-1$, then either $s_k$ commutes with $s_j$ or one can perform the braid move $s_{j+1}s_js_{j+1} = s_js_{j+1}s_j$. 
        \item If $\sigma = s_j \cdots s_i$ for $i<j<n$ and $k \neq j+1$, then either $s_k$ commutes with $s_j$ or one can perform the braid move $s_{j-1}s_js_{j-1} = s_js_{j-1}s_j$.
        \item If $\sigma = s_n \cdots s_i$ and $k \neq n-1$, then $s_k$ commutes with $s_n$. 
        \item  If $\sigma = s_j \cdots s_n \cdots s_i$ for $j \in [n-1]$ and $k \neq j-1$, then either $s_k$ commutes with $s_j$ or one can perform the braid move $s_{j+1}s_js_{j+1} = s_js_{j+1}s_j$.
    \end{itemize}
    In any case, there exists a reduced word of $s_k\sigma$ with a (strict) final subword $v$ such that 
    \[v \notin \{s_j \cdots s_i \mid j \geq i\} \cup \{s_j\cdots s_n \cdots s_i \mid j \in [n-1]\}\]
    or, respectively, 
    \[v \notin \{s_j \cdots s_i \mid j \leq i\} .\]
    But the length of $v$ is at most $\ell(\sigma)$ and $v$ does not satisfy the claim, and so a reduced word for $v$ does not end in $s_{i+1}s_i$ (resp. $s_{i-1}s_i$). Since $v$ is a final subword of $s_k\sigma$, there is a reduced word for $s_k\sigma$ that does not end in $s_{i+1}s_i$ (resp. $s_{i-1}s_i$). In particular, this contradicts the induction, and so we have the claim.
\end{proof}

\begin{proposition}\label{prop:hdes_i}
    Let $i \in [n-2]$. Then 
    \[
    \hdes{i} = \begin{cases}
        \{s_i\} &\text{if }\{t_{i-1},t_{i}\} \subset S(H) \\
        \{s_j\cdots s_i \mid j \leq i\} &\text{if }\{t_{i-1},t_{i}\} \cap S(H) = \{t_{i}\} \\
        \{s_j\cdots s_i \mid j \geq i\} \cup \{s_j\cdots s_n\cdots s_i \mid j \in [n-1]\} &\text{if } \{t_{i-1},t_{i}\} \cap S(H) = \{t_{i-1}\} \\
        \Wi \setminus \{e\} &\text{if } \{t_{i-1},t_{i}\} \cap S(H) = \emptyset.
    \end{cases}
    \]
\end{proposition}
\begin{proof}
    Let $w \in \hdes{i}$. By Lemma \ref{lem:subword_endings}, $\hdes{i} \subseteq \Wi \setminus e$, and either $w=s_i$ or every reduced word for $w$ ends in $s_{i+1}s_i$ or $s_{i-1}s_i$. It is quick to check that $s_{i+1}s_it_i = s_{i} = s_{i-1}s_it_{i-1}$. Lemma \ref{lem:subword_partial_endings} thus gives the containment 
    \[
    \hdes{i} \subseteq \begin{cases}
        \{s_i\} &\text{if }\{t_{i-1},t_{i}\} \subset S(H) \\
        \{s_j\cdots s_i \mid j \leq i\} &\text{if }\{t_{i-1},t_{i}\} \cap S(H) = \{t_{i}\} \\
        \{s_j\cdots s_i \mid j \geq i\} \cup \{s_j\cdots s_n\cdots s_i \mid j \in [n-1]\} &\text{if } \{t_{i-1},t_{i}\} \cap S(H) = \{t_{i-1}\} \\
        \Wi \setminus \{e\} &\text{if } \{t_{i-1},t_{i}\} \cap S(H) = \emptyset.
    \end{cases}
    \]
    Since $s_i$ is the only descent, Lemma \ref{lem:t_positivity_from_s} gives that it is sufficient to check $wt_{i-1}$ and $wt_i$ for containment within $\hdes{i}$, and so the reverse containment is also true.
\end{proof}
\begin{proposition}\label{prop:hdes_n1}
    \[\hdes{n-1} = \begin{cases}
        \{s_{n-1}\} & \text{if } \{t_{n-2},t_{n-1}\} \subset S(H)\\
        \{s_j \cdots s_{n-1} \mid j \in [n-1]\}&\text{if } \{t_{n-2},t_{n-1}\} \cap S(H) = \{t_{n-1}\} \\ 
        \{s_{n}s_{n-1}, s_{n-1}\} &\text{if } \{t_{n-2},t_{n-1},t_n\} \cap S(H) = \{t_{n-2},t_n\} \\
        \{s_{n-1}\} \cup \{s_j\cdots s_ns_{n-1} \mid j \in [n]\} &\text{if } \{t_{n-2},t_{n-1},t_n\} \cap S(H) = \{t_{n-2}\} \\
        \{s_ns_{n-1}\} \cup \{s_j \cdots s_{n-1} \mid j \in [n-1]\}&\text{if } \{t_{n-2},t_{n-1},t_n\} \cap S(H) = \{t_n\} \\ 
        \mathfrak{W}^{n-1} \setminus \{e\}&\text{if } \{t_{n-2},t_{n-2},t_n\} \cap S(H) = \emptyset.
    \end{cases}\]
\end{proposition}
\begin{proof}
    By Lemma \ref{lem:subword_endings}, if $w \in \hdes{n-1}$ and $w \neq s_{n-1}$ then $w$ ends in $s_{n-2}s_{n-1}$ or in $s_ns_{n-1}$. Additionally, $s_{n-2}s_{n-1}t_{n-2} = s_{n-1}= s_ns_{n-1}t_{n-1}$. By Lemma \ref{lem:subword_partial_endings}
    \[\hdes{n-1} \subseteq \begin{cases}
        \{s_{n-1}\} & \text{if } \{t_{n-2},t_{n-1}\} \subset S(H)\\
        \{s_j \cdots s_{n-1} \mid j \in [n-1]\}&\text{if } \{t_{n-2},t_{n-1}\} \cap S(H) = \{t_{n-1}\} \\ 
        \{s_{n-1}\}\cup \{s_j \cdots s_ns_{n-1} \mid j \in [n]\}&\text{if } \{t_{n-2},t_{n-1}\} \cap S(H) = \{t_{n-2}\} \\
        \mathfrak{W}^{n-1} \setminus \{e\}&\text{if } \{t_{n-2},t_{n-2}\} \cap S(H) = \emptyset. 
    \end{cases}\]
    In addition, $s_ns_{n-1}t_n = s_{n-1}s_ns_{n-1}$ but $s_{n-1}s_ns_{n-1}t_n = s_ns_{n-1}$. So if $t_n \in S(H)$ then $w$ cannot end in $s_{n-1}s_ns_{n-1}$. This gives forward direction of containment.\par 
    Again by Lemma \ref{lem:t_positivity_from_s} we need only check if $wt_{n-2} < w$, $wt_{n-1} < w$, and/or $wt_n < w$, and this gives the reverse direction of containment.    
\end{proof}
\begin{proposition}\label{prop:hdes_n}
    \[
    \hdes{n} =  \begin{cases}
        \{s_{n}\} & \text{if } t_n \in  S(H) \\
        \{s_j \cdots s_{n} \mid j \in [n]\}&\text{if } \{t_{n-1},t_n\} \cap S(H) = \{t_{n-1}\} \\ 
        \mathfrak{W}^{n}\setminus\{e\}&\text{if } \{t_{n-1},t_n\} \cap S(H) =\emptyset .
    \end{cases}
    \]
\end{proposition}
\begin{proof}
    By Lemma \ref{lem:subword_endings}, if $w \in \hdes{n}$ and $w \neq s_n$, then $w$ ends in $s_{n-1}s_n$. Now $s_{n-1}s_nt_n = s_n$ and $s_ns_{n-1}s_nt_{n-1} =s_{n-1}s_n$. So if $t_n \in S(H)$ then $\hdes{n} = \{s_n\}$. If $t_{n-1} \in \hdes{n}$ then any reduced word for $w$ cannot end in $s_ns_{n-1}s_n$, so any element of $w$ of length at least $3$ must end in $s_{n-2}s_{n-1}s_n$. By the same logic as in the proof of Lemma \ref{lem:subword_partial_endings}, $w \in \{s_j\cdots s_n \mid j \in [n]\}$. This gives the forward direction of containment.
    Finally, by Lemma \ref{lem:t_positivity_from_s} it is sufficient to check if $wt_{n-1} < w$ and $wt_n < w$, and we have the reverse direction of containment.
\end{proof}
\begin{example}\label{ex:hdes_set}
    Consider $\mathfrak{W}_5$ and let $S(H) = \{t_1,t_5\}$. We will denote reduced words in parentheses, so that $s_is_js_k = (ijk)$. Then 
    \begin{align*}
        \hdes{1} &=  \{(1)\}\\
        \hdes{2} &=  \left\{\begin{matrix}(12345432), (2345432), ( 345432),\\  (45432), (5432), (432), (32), (2)\end{matrix}\right\}\\
        \hdes{3} &= \left\{\begin{matrix}
        (234512345342312), (34512345342312),  (3452345342312), \\(4512345342312),  (452345342312), (512345342312), \\(12345342312), (45345342312), (52345342312),\\
        (1234542312), (2345342312), (5345342312), (123454312), (234542312), \\
        (345342312), (545342312), (12345432), (23454312), (34542312),\\ (45342312), (2345432), (3454312), (4542312), (5342312), \\(342312), (345432), (454312), (542312), (42312), (45432),\\ (54312), (2312), (4312), (5432), (312), (432), (12), (32), (2)\end{matrix}\right\} \\
        \hdes{4} &=  \{(54),(1234),(234),(34),(4)\}\\
        \hdes{5} &=  \{(5)\}.
    \end{align*}
\end{example}
The one-line notation for some of these elements of $\hdes{i}$ are 
\[s_j\cdots s_i =
\begin{cases}
     [1,\ldots,j-1,j+1,\ldots,i+1,j,i+2,\ldots,n] &\text{if } j<i<n\\
    [1,\ldots,j-1,j+1,\ldots,n,\bj] &\text{if }   j<i = n\\
\end{cases}
\]
and 
\[s_j \cdots s_n\cdots s_i = \begin{cases}
    [1,\ldots,j-1,j+1,\ldots,i,\bj,i+1,\ldots,n] &\text{if } j<i\\
    [1,\ldots,i-1,\bi,i+1,\ldots,n] &\text{if } j=i\\
    [1,\ldots,i-1,\bj,i,\ldots,j-1,j+1,\ldots,n] &\text{if } i<j\leq n.
\end{cases}\]
The one-line notation for elements of $\Wi$ are more complicated. It is however quite easy to tell whether two elements of $\Wn$ are in the same coset of $\mathfrak{S}_i \times \mathfrak{W}_{n-i}$, as 
\[
v \in w \left(\mathfrak{S}_i \times \mathfrak{W}_{n-i}\right)\;\;\text{ if and only if }\;\; v([i]) = w([i]).
\]
\section{Splines in \texorpdfstring{$\splines{H}^1$}{splinesMH1}}\label{sec:linear_splines}
This section defines four families of functions from $\Wn$ to $\poly$: the $\spline{f}$-, $\spline{y}$-, $\spline{g}$-, and $\spline{h}$-splines. It begins with the definitions and characterization of which $H$ these families are members of $\splines{H}$, then gives some linear relations, and finally computes the dot action on them. \par 

A subset $A \subset [\overline{n}]$ is \emph{unbalanced} if $i \in A$ implies $ \bi \notin A$. In particular, any $w(I)$ for $I \subset[n]$ and $w \in \Wn$ is imbalanced and vice versa. For $i \in [n]$, define
\begin{align*}
    \spline{f}_i^A(w) &\coloneqq \begin{cases}
    x_{w(i)}-x_{w(i+1)} &\text{if }w([i]) = A  \text{ and } i < n \\
    x_{w(n)} &\text{if }w([i]) = A  \text{ and } i = n \\
    0 &\text{otherwise}
    \end{cases}
\end{align*}
for all $\abs{A} = i$ and $A$ unbalanced. \par 
Call $s_i$ \emph{uncovered} if either
    \begin{itemize}
        \item $i \neq n-1$ and $\{t_{i-1},t_{i}\} \cap S(H) = \emptyset$, or 
        \item $i = n-1$ and $\{t_{n-2},t_{n-1},t_n\} \cap S(H) = \emptyset$.
    \end{itemize}

\begin{lemma}
    If $s_i$ is uncovered, then $\spline{f}_i^A$ is a spline for all unbalanced $A$ where $\abs{A} = i$.
\end{lemma}
\begin{proof}
    Let $w \in \Wn$ and $s \in S(H)$. If $w([i]) = ws([i])$ then either both $w$ and $ws$ are supported by $\spline{f}_i^A$ or both are not. If not, then $\spline{f}_i^A(w) - \spline{f}_i^A(ws) = 0 \in \llangle \cl(w,ws)\rrangle$. If both $w$ and $ws$ are in the support then 
    \[
    \spline{f}_i^A(w) - \spline{f}_i^A(ws) = \begin{cases}
    \left[\spline{r}_i-\spline{r}_{i-1}\right](w) - \left[\spline{r}_i-\spline{r}_{i-1}\right](ws) &\text{if }i < n \\
    \spline{r}_n(w)-\spline{r}_{n}(ws) &\text{if } i = n, 
    \end{cases}
    \]
    which is in $\llangle \cl(w,ws)\rrangle$ by Lemma \ref{lem:t_and_r}. \par
    If $w([i]) = A$ and $ws([i]) \neq A$, then $s = s_i$. In particular, \[\spline{f}_i^A(w) -\spline{f}_i^A(ws_i) = \spline{f}_i^A(w) = \left\{\begin{matrix}
        x_{w(i)}-x_{w(i+1)} &\text{if } i<n\\
        x_{w(n)} &\text{if } i=n\\
    \end{matrix}\right\} \in \llangle \cl(w,ws_i)\rrangle,\] and so $\spline{f}_i^A$ is a spline.
\end{proof}

Define
\[\spline{y}_{i,k}(w) \coloneqq \begin{cases} 
x_k - x_{w(i+1)} &\text{ if }w^{-1}(k) \in [i]\\
0&\text{ otherwise}.
\end{cases}\]
for $k \in [\overline{n}]$ and $i \in [n-1]$.
\begin{lemma}
    If either \begin{itemize}
        \item $i \in [n-2]$ and $t_i \notin S(H)$, or 
        \item $i = n-1$ and $\{t_{n-1},t_n\} \cap S(H) = \emptyset$, then 
    \end{itemize} 
    $\spline{y}_{i,k}$ is a spline for all $k \in [\bn]$.
\end{lemma}
\begin{proof}
    Let $w \in \Wn$ and $s \in S(H)$. If $s w^{-1}(k) \in [i]$ and $ w^{-1}(k) \in [i]$ then both $w$ and $ws$ are in the support of $\spline{y}_{i,k}$ and so the difference \[\spline{y}_{i,k}(ws) - \spline{y}_{i,k}(w) = \left[\spline{t}_k-\spline{r}_{i+1}\right](w)-\left[\spline{t}_k-\spline{r}_{i+1}\right](ws) \in \llangle \cl(w,ws)\rrangle.\] Similarly, if both are not supported then the same occurs. \par 
    Now if $s w^{-1}(k) \notin [i]$ and $ w^{-1}(k) \in [i]$, that means that $w^{-1}(k) \notin s([i])$ and $w^{-1}(k) \in [i]$. In particular, $s=s_i$ or $s=t_{i-1}$. \par 
    If $s = s_i$ then $w^{-1}(k) = i$. So $w(i) = k$ and $\spline{y}_{i,k}(w)  = t_{w(i)}-t_{w(i+1)} \in \llangle \cl(w,ws)\rrangle$. \par 
    If $s = t_{i-1} = (i-1,i+1)$ then $w^{-1}(k) = i-1$. So $\spline{y}_{i,k}(w) = t_{w(i-1)} - t_{w(i+1)} = \llangle \cl(w,ws)\rrangle$.
\end{proof}
  Define 
\[
\spline{g}_i \coloneqq \begin{cases}
    x_i & \text{if } w^{-1}(i) < 0 \\
    0   &\text{otherwise.}
\end{cases}
\]
for $i \in [\overline{n}]$.
\begin{lemma}
    If $t_{n}\notin S(H)$ then $\spline{g}_i$ is a spline for all $i \in [\bar{n}]$.
\end{lemma}
\begin{proof}
    Let $w \in \Wn$ and $s \in S(H)$. If both $w$ and $ws$ are in the support or not in the support of $\spline{g}_i$, then $\spline{g}_i(w) - \spline{g}_i(ws) = 0$ and containment is trivial.

    Say $w$ is in the support and $ws$ is not. The only reflections in $S(H)$ that negate elements are $s_n = (n,\bn)$, $t_{n-1} = (n-1,\overline{n-1})$, and $t_n = (n-1,\bn)$. The condition is that $(n-1,\bn) \notin S(H)$. So we need to check the values on the pairs $(w,w(n,\bn))$ and $(w,w(n-1,\overline{n-1}))$. If $w^{-1}(i) < 0$ and $(n,\bn)w^{-1}(i)>0$ then $i = w(n)$. In particular, $x_i = x_{w(n)}$. The same holds for $(n-1,\overline{n-1})$.
\end{proof}

Finally, recall that $\Neg(w) \coloneqq \{w(i) \mid i \in [n],\; w(i) < 0\}$, and define
\[
\spline{h}(w) \coloneqq \begin{cases}
    x_{w(n)} &\text{if }\abs{\Neg(w)} \text{ is odd} \\
    0 &\text{otherwise.}
\end{cases}
\]

\begin{lemma}
    If $t_{n-1} \notin S(H)$ then $\spline{h}$ is a spline.
\end{lemma}
\begin{proof}
    Let $w \in \Wn$ and $s \in S(H)$. If both $w$ and $ws$ are in the support of $\spline{h}$, then $\spline{h}(w) - \spline{h}(ws) = \spline{r}_n(w) - \spline{r}_n(ws) \in \llangle \cl(w,ws)\rrangle$, and if both are unsupported then containment is trivial.

    Say $w$ is in the support and $ws$ is not. Recall $t_{n-1} = s_{n-1}s_ns_{n-1} = (n-1,\overline{n-1})$. So the only reflection in $S(H)$ that can change the parity of $\abs{\Neg(w)}$ is $s_n = (n,\overline{n})$. Since $\spline{h}(w) - \spline{h}(ws_n) = \spline{h}(w) = x_{w(n)} \in \llangle \cl(w,ws_n)\rrangle$, we have the claim. 
\end{proof}

There are several linear relations between these families.
\begin{proposition}\label{prop:relations}
    Let $\spline{t}_i$, $\spline{r}_i$, $\spline{f}_i^v$, $\spline{y}_{i,k}$, and $\spline{g}_i$ be as above. Then
    \begin{itemize}
        \item[\textsc{(1)}] ${\ds \;\; \sum_{A \subset [\bn]} \spline{f}_i^A = \spline{r}_i -\spline{r}_{i+1}}$ for all $i \in [n-1]$ and ${\ds \;\; \sum_{A \subset [\bn]} \spline{f}_n^A = \spline{r}_n}$,
        \item[\textsc{(2)}] ${\ds \;\;\sum_{k\in [\bn]} \spline{y}_{i,k} = \sum_{j \in [i]} \spline{r}_{j}-i\spline{r}_{i+1}}$ for all $i \in [n-1]$,
    \end{itemize}
    if $p < m < n$ then 
    \begin{itemize}
        \item[\textsc{(3)}] ${\ds \;\;\spline{y}_{p,k} + \sum_{p<i \leq m}\sum_{A \ni k} \spline{f}_i^A = \spline{y}_{m,k}  }$ for all $k \in [\overline{n}]$, 
    \end{itemize}
    if $p <n$, then 
    \begin{itemize}
        \item[\textsc{(4)}] ${\ds\;\; \spline{y}_{p,k} + \sum_{p<i\leq n}\sum_{A \ni k} \spline{f}_i^A = -\spline{g}_{\bar{k}}}$,
    \end{itemize}
    and if $i \in [n]$ then 
    \begin{itemize}
        \item[\textsc{(5)}] ${\ds\;\; \spline{g}_{i} - \spline{g}_{\bar{i}} = \spline{t}_i }$ and $\ds 2\sum_{j \in [n]} \spline{g}_j = \sum_{j \in [n]} \spline{t}_j-\spline{r}_j$.
    \end{itemize}
\end{proposition}
\begin{proof}
    The first relation is straightforward. \par 
    Consider the second relation, and pick an arbitrary $w \in \Wn$. Then 
    \begin{align*}
        \sum_{k \in [\overline{n}]} \spline{y}_{i,k}(w) &= \sum_{k \in [\overline{n}]}\begin{cases}
            x_k-x_{w(i+1)} &\text{if } w^{-1}(k) \in [i] \\
            0 & \text{otherwise}
        \end{cases} \\
        &= \sum_{k \in \{w(1),\ldots ,w(i)\} } \left(x_k-x_{w(i+1)} \right)\\
        &= \sum_{j \in [i]}x_{w(j)} - ix_{w(i+1)} \\
        &= \sum_{j \in [i]} \spline{r}_{j}(w)-i\spline{r}_{i+1}(w).
        \end{align*} 
    As for the third relation, pick an arbitrary $w \in \Wn$ and proceed based on cases for $k$.\\
    If $k \in w([p])$ then 
    \begin{align*}
        \spline{y}_{p,k}(w) + \sum_{p<i < m}\sum_{A \ni k} \spline{f}_i^A(w) &= x_k-x_{w(p+1)} + \sum_{p<i<m} x_{w(i)}-x_{w(i+1)} \\
        &= x_k - x_{w(m+1)}.
    \end{align*}
    If $k \in w(\{p+1,\ldots ,m\})$ then $\spline{y}_{p,k}(w) = 0$, and 
    \[ \sum_{p<i < m}\sum_{A \ni k} \spline{f}_i^A(w) = \sum_{w^{-1}(k)\leq i < m}\sum_{A \ni k} \spline{f}_i^A(w) = x_k  - x_{w(m+1)}.\]
    Finally, if $k \notin w([m])$ then $w$ is not in the support of any spline on either side, and the claim follows. \par 
    For the fourth relation, also pick an arbitrary $w \in \Wn$ and proceed base on cases for $k$. If $k \in w([p])$ then 
    \begin{align*}
        \spline{y}_{p,k}(w) + \sum_{p<i\leq n}\sum_{A \ni k} \spline{f}_i^A(w) &= x_k - x_{w(p+1)} + \left(\sum_{p<i<n} x_{w(i)} - x_{w(i+1)}\right) +x_{w(n)} \\
        &= x_k.
    \end{align*}
    If $k \in \{w(p+1),\ldots ,w(n)\}$ then $\spline{y}_{p,k}(w) = 0$, and 
    \[ \sum_{p<i \leq n}\sum_{A \ni k} \spline{f}_i^A(w) = \sum_{w^{-1}(k)\leq i < n}\sum_{A \ni k} \spline{f}_i^A(w) + x_{w(n)} = x_k.\]
    On the other hand if $k \notin w([n])$ then $w$ is not in the support of any spline in the sum and so it evaluates to zero. So
    \begin{align*}
    \spline{y}_{p,k}(w) + \sum_{p<i < m}\sum_{A \ni k} \spline{f}_i^A(w) &= \begin{cases}
        x_k & w^{-1}(k) >0 \\
        0 & w^{-1}(k) <0
    \end{cases}\\
    &= -\spline{g}_{\bar{k}}.
    \end{align*}
    For the last relations, we directly compute 
    \[
    \spline{g}_i(w) - \spline{g}_{\overline{i}}(w) = \begin{cases}
        x_i &\text{if }w^{-1}(i) < 0 \\
        -x_{\overline{i}} &\text{if }w^{-1}(i) > 0
    \end{cases} = \spline{t}_i(w)
    \]
    and 
    \begin{align*}
    2\sum_{j \in [n]} \spline{g}_j (w) &= 2\sum_{j \in [n]}\left.\begin{cases} x_j &\text{if }w^{-1}(j) < 0 \\ 0 &\text{otherwise} \end{cases}\right\}
     = 2\sum_{j \in \Neg(w)}x_j 
     = \sum_{j \in [n]} \spline{t}_j(w)-\spline{r}_j(w).
    \end{align*}
\end{proof}

One may let $p=0$ in Proposition \ref{prop:relations}(3), so $\spline{y}_{p,k} = 0$, and the proof still holds. \par 

These families of splines interact well with the dot action.
\begin{lemma}\label{lem:dotact_most}
    Let $w \in \Wn$. Then 
    \[ w\cdot \spline{t}_i = \spline{t}_{w(i)},\;\;\; 
    w \cdot \spline{r}_i = \spline{r}_{i},\;\;\; w \cdot \spline{f}_i^A = \spline{f}_{i}^{w(A)},\;\;\; w \cdot \spline{y}_{i,k} = \spline{y}_{i,w(k)},\;\text{and }w \cdot \spline{g}_i = \spline{g}_{w(i)}.\]
\end{lemma}
\begin{proof}
    We evaluate each of the claimed equalities at an arbitrary $u \in \Wn$. The first two are straightforward. As for the others, we compute directly:
    \begin{align*}
    w \cdot \spline{f}_i^A(u) &= w\spline{f}_i^A(w^{-1}u) \\
        &= \begin{cases}
            w(x_{w^{-1}u(i)}-x_{w^{-1}u(i+1)}) &\text{if }w^{-1}u([i]) = A  \text{ and } i < n \\
            w(x_{w^{-1}u(n)}) &\text{if }w^{-1}u([i]) = A  \text{ and } i = n \\
            0 &\text{otherwise.}
            \end{cases}\\
        &= \begin{cases}
            x_{u(i)}-x_{u(i+1)} &\text{if }u([i]) = w(A)  \text{ and } i < n \\
            x_{u(n)} &\text{if }u([i]) = w(A)  \text{ and } i = n \\
            0 &\text{otherwise.}
            \end{cases}\\
        &= \spline{f}_i^{w(A)}(u) \\        
    w \cdot \spline{y}_{i,k}(u) &= w \spline{y}_{i,k} (w^{-1}u) \\
        &=  \begin{cases} 
        w(x_k - x_{w^{-1}u(i+1)})&\text{ if }(w^{-1}u)^{-1}(k) \in [i]\}\\
        0&\text{ otherwise}.
        \end{cases}\\
        &=  \begin{cases} 
        x_{w(k)} - x_{u(i+1)})&\text{ if }u^{-1}w(k) \in [i]\\
        0&\text{ otherwise}.
        \end{cases}\\
        &= \spline{y}_{i,w(k)}(u)\\
    w \cdot \spline{g}_i (u) &=  w\spline{g}_i (w^{-1}u) \\
         &= \begin{cases}
        w(x_i) & \text{if } (w^{-1}u)^{-1}(i) < 0 \\
        0   &\text{otherwise.}
        \end{cases} \\
        &= \begin{cases}
        x_{w(i)} & \text{if } u^{-1}w(i) <0 \\
        0   &\text{otherwise.}
        \end{cases}\\
        &=  \spline{g}_{w(i)} (u).
    \end{align*}
\end{proof}
The $\spline{h}$-spline is not quite as well-preserved by the dot action, but with a slight alteration can be made so.
\begin{lemma}\label{lem:dotact_h}
    Let $w \in \Wn$. Then 
    \[w \cdot \spline{h} = \begin{cases}
        \spline{r}_n-\spline{h} &\text{if }\abs{\Neg(w)}\text{ is odd,}\\
        \spline{h} &\text{if }\abs{\Neg(w)}\text{ is even.}
    \end{cases}\]
    Moreover, 
        \[w \cdot\left(\spline{r}_n-2\,\spline{h}\right) = \begin{cases}
        2\,\spline{h}-\spline{r}_n &\text{if }\abs{\Neg(w)}\text{ is odd,}\\
        \spline{r}_n-2\,\spline{h} &\text{if }\abs{\Neg(w)}\text{ is even.}
    \end{cases}\]
\end{lemma}
\begin{proof}
    This computation follows from two facts. First, the parity of $\abs{\Neg(wv)}$ is equal to the parity of $\abs{\Neg(w)}\cdot \abs{\Neg(v)}$. Second, 
    \[\left[\spline{r}_n-\spline{h}\right](v) = \begin{cases}
    x_{w(n)} &\text{if }\abs{\Neg(w)} \text{ is even} \\
    0 &\text{otherwise.}
\end{cases}\]
The “moreover” part is linear algebra, and is included because it makes the corresponding one-dimensional representation of $\Wn$ more clear.
\end{proof}

\section{Generators for \texorpdfstring{$\splines{H}^1$}{SplinesH1}}\label{sec:generators}
First, this section argues that the $\spline{f}$-, $\spline{y}$-, $\spline{g}$-, and/or $\spline{h}$-splines from Section \ref{sec:linear_splines} form a $\C$-generating set for $\splines{H}^1$. Then, relations from Proposition \ref{prop:relations} are used to reduce the number of generators in preparation for building different bases in Section \ref{sec:bases_reps}. \par 
\begin{lemma}\label{lem:shortelts}
    The set of shortest elements in the supports of the splines\begin{center}
    \renewcommand{\arraystretch}{1.75}
    \begin{tabular}{rlcl}
        (1) & $\ds \left\{\sum_{i=1}^k \spline{r}_i-\spline{t}_i\right\}$ & is & $\{s_k\}$, \\
        (2) & $\ds\left\{\spline{f}_i^A \mid \abs{A} = i,\;\; A\text{ is unbalanced}\right\}$ & is & $\Wi$,\\
        (3) & $\ds\left\{\spline{y}_{i,k} \mid k \in [\bn] \setminus[i] \right\}$ & is & $\{s_{k-1} \cdots s_i \mid k > i\}\cup \{s_{\abs{k}}\cdots s_n\cdots s_i \mid k<0\}$, \\
        (4) & $\ds \left\{\spline{t}_k - \spline{r}_{i+1} - \spline{y}_{i,k} \mid k \in [i]\right\}$ & is & $\{s_k \cdots s_{i+1} \mid k \in [i]\}$,\\
        (5) & $\ds \left\{\spline{g}_i \mid i \in [n]\right\}$  & is & $\{s_i \cdots s_n \mid i \in [n]\}$, and\\
        (6) & $\ds \left\{\spline{h}-\frac{1}{2}\sum_{i \in [n]} (\spline{r}_i - \spline{t}_i)\right\}$ & is & $\{s_ns_{n-1}\}$.\\
    \end{tabular}
    \end{center}
\end{lemma}
\begin{proof}
    We leave (1) to the reader. The splines in (2) are supported on cosets of $\mathfrak{S}_i \times \mathfrak{W}_{n-i}$, and so have shortest support element $\Wi$ by definition. \par 
    The shortest element in the support of $\spline{y}_{i,k}$ is the shortest element $w$ such that $w^{-1}(k) \in [i]$. If $k \in [\bn]\setminus [i]$, then $w^{-1} = s_i \cdots s_{k-1}$ if $k>0$ and $s_i \cdots s_n \cdots s_k$ if $k<0$. The inverses are $s_{k-1} \cdots s_i$ and $s_k\cdots s_n\cdots s_i$ respectively, and (3) follows. \par 
    Consider $k \in [i]$ and the shortest element $w$ in the support of $\spline{t}_k - \spline{r}_{i+1} - \spline{y}_{i,k}$. Surely $w^{-1}(k) \notin [i]$, but \textbf{also} $k \neq w(i+1)$. In particular, we need $w^{-1}(i) \notin [i+1]$. So the shortest element in the support must have inverse $w^{-1} = s_{i+1}\cdots s_k$, and we have (4). \par 
    The support of $\spline{g}_i$ for $i \in [n]$ are those $w \in \Wn$ such that $w^{-1}(i) < 0$. Since $s_n$ is the only simple transposition that negates elements in $w(n)$, the shortest element $w$ in the support of $\spline{g}_i$ must have $w(n) = \bi$. The shortest element that satisfies this condition is $s_i \cdots s_n$, which is (5). \par 
    The spline $\spline{h}$ is supported at $s_n$, $s_ns_{n-1}$, $s_{n-1}s_n$, and no other elements of length $\leq 2$. The same is true for $\sum \spline{t}_i-\spline{r}_i$. Now \[-\spline{h}(s_n) = x_n = \frac{1}{2}\sum(\spline{t}_i-\spline{r}_i)(s_n)\]
    and 
    \[-\spline{h}(s_{n-1}s_n) = x_{n-1} = \frac{1}{2}\sum(\spline{t}_i-\spline{r}_i)(s_{n-1}s_n)\]
    but 
    \[\spline{h}(s_ns_{n-1}) = x_{n-1} \text{ and } \frac{1}{2}\sum(\spline{t}_i-\spline{r}_i)(s_ns_{n-1}) = x_{n}\]
    So the difference  has shortest support $s_ns_{n-1}$, and we have (6).
\end{proof}
\begin{definition}\label{def:itypes}
Consider $S(H)$ and $i \in [n]$. We call $i$\begin{itemize}
    \item \emph{uncovered} if $\begin{cases} \{t_{i-1},t_i\} \cap S(H) = \emptyset &\text{and } i \neq n-1\text{, or} \\ \{t_{n-1},t_{n-1},t_{n}\} \cap S(H) = \emptyset &\text{and } i = n-1\end{cases}$
    \item \emph{surrounded} if $i \in [n-2]$ and $\{t_{i-1},t_i\} \cap S(H) = \{t_{i-1}\}$ and $t_m \in S(H)\text{ for some } m>i$,
    \item \emph{shaded} if $t_i \in S(H)$ or if $t_n \in S(H)$ and $i=n-1$.
\end{itemize}
\end{definition}
Let 
\begin{equation}\label{eq:genset}
    \begin{aligned}
\spline{T} &\coloneqq \{\spline{t}_i \mid i \in [n]\} \\
\spline{R} &\coloneqq \{\spline{r}_i \mid i \in [n]\} \\
\spline{F}_H &\coloneqq \{\spline{f}_i^A \mid i \text{ is uncovered, }\abs{A} = i\text{, and $A$ is unbalanced}\} \\
\spline{Y}_H &\coloneqq \left\{\spline{y}_{i,k} \mid i\text{ is surrounded and }k\in [\bn] \right\} \\
\spline{G}_H &\coloneqq \begin{cases}
    \{\spline{g}_i \mid i \in [n] \} &\text{if $t_n \notin S(H)$}\\
    \{\spline{h} \} &\text{if }\{t_{n-1},t_n\} \cap S(H) = \{t_{n}\}\\
    \emptyset &\text{otherwise.}
\end{cases}
\end{aligned}
\end{equation}

\begin{proposition}\label{prop:generators}
    The set of splines 
    \[
    \spline{T} \cup \spline{R}\cup\spline{F}_H\cup 
\spline{Y}_H\cup \spline{G}_H
    \]
    from \eqref{eq:genset} is a $\C$-generating set for $\splines{H}^1$.
\end{proposition}
\begin{proof}
    The set of elements in $\hdes{i}$ for $i \in [n]$ are given in Propositions \ref{prop:hdes_i}, \ref{prop:hdes_n1}, and \ref{prop:hdes_n}. By Lemma \ref{lem:shortelts}, the set of shortest support elements for the following set of splines contains $\hdes{i}$ for all $i \in [n]$:
    \begin{align*}
        \spline{F}_H &\cup \left\{\sum_{i=1}^k \spline{r}_i-\spline{t}_i \mid i \in [n]\right\}  \\
        &\cup \left\{\spline{y}_{i,k} \mid i \in [n-2],\;\;k \in [\bn],\;\; \{t_{i-1},t_i\} \cap S(H) = \{t_{i-1}\}    \right\} \\
        &\cup \left\{\spline{y}_{n-1,k} \mid k \in [\bn],\;\; \{t_{n-2},t_{n-1},t_n\} \cap S(H) = \{t_{n-2}\}    \right\} \\
        &\cup \left\{\spline{t}_k - \spline{r}_i - \spline{y}_{i-1,k} \mid i \in [n-2],\;\;k\in [\bn],\;\; \{t_{i-1},t_i\} \cap S(H) = \{t_{i}\}    \right\} \\
        &\cup \left\{\spline{t}_k - \spline{r}_{n-1} - \spline{y}_{n-2,k} \mid k\in [\bn],\; \emptyset \neq \{t_{n-2},t_{n-1},t_n\} \cap S(H) \subseteq \{t_{n-1},t_n\}     \right\} \\
        &\cup \{\spline{g}_i \mid i \in [n],\;\; t_n \notin S(H)\}\\
        &\cup \{\spline{h} \mid \{t_{n-1},t_n\} \cap S(H) = \{t_{n}\} \}.
    \end{align*}
    So these elements span $\splines{H}^1$. They are contained within the span of the following set:
    \begin{align*}
        \spline{T} \cup \spline{R} &\cup \spline{F}_H \\
        &\cup \left\{\spline{y}_{i,k} \mid i \in [n-2],\;\;k\in [\bn],\;\; \{t_{i-1},t_i\} \cap S(H) = \{t_{i-1}\}    \right\} \\
        &\cup \left\{\spline{y}_{n-1,k} \mid k \in [\bn],\;\; \{t_{n-2},t_{n-1},t_n\} \cap S(H) = \{t_{n-2}\}    \right\} \\
        &\cup \left\{\spline{y}_{i-1,k} \mid i \in [n-1],\;\;k\in [\bn],\;\; \{t_{i-1},t_i\} \cap S(H) = \{t_{i}\}    \right\} \\
        &\cup \left\{\spline{y}_{n-2,k} \mid k\in [\bn],\; \emptyset \neq \{t_{n-2},t_{n-1},t_n\} \cap S(H) \subseteq \{t_{n-1},t_n\}     \right\} \\
        &\cup \{\spline{g}_i \mid i \in [n],\;\; t_n \notin S(H)\}\\
        &\cup \{\spline{h} \mid \{t_{n-1},t_n\} \cap S(H) = \{t_{n}\} \}.
    \end{align*}
    If $i \in [n-2]$ and $\{t_{i-1},t_i\} \cap S(H) = \{t_{i}\}$ or $i=n-1$ and $\emptyset \neq \{t_{n-2},t_{n-1},t_n\} \cap S(H) \subseteq \{t_{n-1},t_n\}$, then either there exists some $p < i$ such that $\{t_{p-1},t_p\} \cap S(H) = \{t_{p-1}\}$ or $t_p \notin S(H)$ for all $p < i$. In both cases, the corresponding set $\left\{\spline{y}_{i-1,k} \mid k\in [\bn] \right\}$ may be removed without affecting the span by Proposition \ref{prop:relations}(3). \par 
    If $p \in [n-2]$ and $\{t_{p-1},t_p\} \cap S(H) = \{t_{p-1}\}$, or if $p=n-1$ and $\{t_{n-2},t_{n-1},t_n\} \cap S(H) = \{t_{n-2}\}$, and $t_m \notin S(H)$ for all $m > p$, then $t_n \notin S(H)$. In particular, $\spline{g}_i$ is in the set and so by Proposition \ref{prop:relations}(4,5) discard $\left\{\spline{y}_{p,k} \mid k\in [\bn] \right\}$. After discarding these particular $\spline{y}$-splines from our generating set, the following remains: 
    \begin{align*}
        \spline{T} \cup \spline{R} \cup \spline{F}_H &\cup \left\{\spline{y}_{i,k} \left| \begin{matrix}i \in [n-2],\;\;k\in [\bn],\;\; \{t_{i-1},t_i\} \cap S(H) = \{t_{i-1}\},\\ \exists m>i \text{ where }t_m \in S(H) \end{matrix}  \right.\right\} \\
        &\cup \{\spline{g}_i \mid i \in [n],\;\;t_n \notin S(H)\}\\
        &\cup \{\spline{h} \mid \{t_{n-1},t_n\} \cap S(H) = \{t_{n}\} \}.
    \end{align*}
    This is precisely $\spline{T} \cup \spline{R}\cup\spline{F}_H\cup 
\spline{Y}_H\cup \spline{G}_H$
\end{proof}
It can be helpful to associate the
\begin{itemize}
    \item $\spline{f}$-splines to uncovered elements
    \item $\spline{y}$-splines to intervals $[i < j]$ where $i \in [n-2]$ $\{t_{i-1},t_j\} \subset S(H)$ but $t_p \notin S(H)$ for all $i<p<j$,
    \item $\spline{g}$-splines to right-open intervals $(i,n)$ where $t_i \in S(H)$ but $t_p \notin S(H)$ for all $i<p\leq n$, and
    \item the  $\spline{h}$-spline to the specific case where $t_{n-1} \notin S(H)$ but $t_n \in S(H)$.
\end{itemize}

\section{Bases of \texorpdfstring{$\splines{H}^1$}{MH1} and the Left \& Right Representations}\label{sec:bases_reps}

The left representation $\lqot{H}$ is achieved by computing the dot action after quotienting the $\C$-vector space $\splines{H}^1$ by $\C\spline{T}$, and the right representation $\rqot{H}$ in the same manner by $\C\spline{R}$. We build two bases for $\splines{H}^1$ that are conducive to both the dot action on $\splines{H}^1$ as well as the quotients $\lqot{H}$ and $\rqot{H}$.\par
First, consider the special case that $H = \Delta$, as it is easier to separate. Furthermore, it is already known by other means \cite{Stembridge_Permuto}.
\begin{lemma}\label{lem:permuto_basis}
    If $H = \Delta$ then $\spline{F}_H$ is a basis for $\splines{H}^1$.
\end{lemma}
\begin{proof}
    The set $\spline{T} \cup \spline{F}_H$ is a generating set, and since $\left\{\spline{f}_i^{[i]}(e) \mid i \in [n]\right\} = \{x_i - x_{i+1},x_n \mid i \in [n-1]\}$ generates $\llangle x_1,\ldots,x_n\rrangle$, $\spline{F}_H$ is a generating set. Its size is $\sum_{i=1}^n \abs{\Wi} = n + \sum_{i=1}^n \left(\abs{\Wi}-1\right)$, and so it is a basis.
\end{proof}
\begin{proposition}\label{prop:permutohedral}
    Let $\bm{\chi}$ be the character of the defining representation of $\Wn$ (i.e. the action on $\C^n$). Let $\fh_i$ be the character of the action of $\Wn$ on the cosets of $\mathfrak{S}_i \times \mathfrak{W}_{n-i}$. Let $\mathbb{1}$ be the character of the trivial representation. Then 
   \[
     \lrep{\Delta}_1 = \sum_{i=1}^n \fh_i - \bm{\chi} 
    \hspace{1cm}\text{and}\hspace{1cm}
    \rrep{\Delta}_1 = \sum_{i=1}^n \fh_i - n\,\mathbb{1}.
    \]
\end{proposition}
\begin{proof}
    The representation on the subspace $\C \spline{T}$ has character $\bm{\chi}$ and the representation on the subspace $\C \spline{R}$ has character $n\,\mathbb{1}$.
\end{proof}
We note that $\fh_1-\bchi = \fs$
\begin{remark}
    The case $H = \Delta$ corresponds to the geometry of the \emph{permutohedral variety}. The natural generalization of $\spline{F}_H$ to cosets of all Young subgroups gives a generating set for $\splines{\Delta}$ in all degrees. It would be interesting to find a basis of $\splines{\Delta}$ in all degrees that is conducive to the dot action, as has been done in type A \cite{ChoHongLee_Bases}.
\end{remark}
For the remainder of the section, assume that $H \neq \Delta$. Since Theorem \ref{thm:character} is referenced by Theorem \ref{intthm:character}, we will restate this assumption in Theorem \ref{thm:character}. \par 
By Proposition \ref{prop:generators}, 
\[\spline{T} \cup \spline{R} \cup \spline{F}_H \cup \spline{Y}_H \cup \spline{G}_H\]
is a generating set for $\splines{H}^1$. 
\begin{lemma}\label{lem:left_basis}
    The set 
    \[
    \spline{T} \cup \left\{\spline{r}_i \mid i \text{ is shaded}\,\right\} \cup \spline{F}_H \cup \spline{Y}_H \cup \spline{G}_H
    \]
    is a basis for $\splines{H}^1$.
\end{lemma}
\begin{proof}
    By Proposition \ref{prop:relations}(3), if $t_i \notin S(H)$ then every spline in $\{\spline{y}_{i,k} \mid k \in [\bn]\}$ is in the span of $\spline{F}_H\cup \spline{Y}_H$. By Proposition \ref{prop:relations}(2) $\spline{r}_i$ is also in the span, and so each such $\spline{r}_i$ may be removed. It follows that the proposed set is in fact a spanning set for $\splines{H}^1$. \par 
    We show this is a basis by dimension. Recall from Proposition \ref{prop:dim1} that the dimension of $\splines{H}^1$ is $\ds n + \sum_{i=1}^n \abs{\hdes{i}}$. Since $\abs{\spline{T}} = n$, it remains to show that 
    \[
    \abs{\left\{\spline{r}_i \mid i \text{ is shaded}\right\}} \cup \abs{\spline{F}_H} \cup \abs{\spline{Y}_H} \cup \abs{\spline{G}_H} = \sum_{i=1}^n \abs{\hdes{i}}.
    \]
    Each of the splines in this set are associated with a particular simple reflection $s_i$:
    \[\spline{f}_i^A \longleftrightarrow s_i,\hspace{1cm}\spline{y}_{i,k} \longleftrightarrow s_i,\hspace{1cm}\spline{g}_i \longleftrightarrow s_{n-1},\hspace{1cm}\spline{h} \longleftrightarrow s_n. \] We will show that the sum of $\hdes{i}$ and the number of associated splines are equal over certain intervals $\{p,\ldots ,m\}$ in $[n]$.\par 
    (Interval 1: totally covered) Consider $i \in [n]$ such that $\{t_{i-1},t_i\} \in S(H)$. For each of these single-element intervals, $\hdes{i} = \{s_i\}$ and these elements are naturally associated with the spline $\spline{r}_i$. In particular, 
    \[
    \abs{\left\{\spline{r}_i \mid i \in [n],\; \{t_{n-1},t_i\} \subset S(H) \right\}} = \sum_{\substack{i \in [n] \\ \{t_{i-1},t_i\} \subset S(H)}} \hdes{i}.
    \]
    (Interval 2: left-open) If $m \in [n]$, $t_i \notin S(H)$ for $i \in [m-1]$, and $t_{m} \in S(H)$, then the elements in  $\{1,\ldots ,m\}$ are associated to the splines 
    \[\{\spline{f}_i^A \mid i \in [m-1], \;\; \abs{A} = i,\text{ and $A$ is unbalanced}\} \cup \{\spline{r}_m\}.\] This set has size $\ds 1 + \sum_{i=1}^{m-1} \abs{\Wi}$. On the other hand, 
    \[
    \sum_{i=1}^m \hdes{i} = \sum_{i=1}^{m-1} \left(\abs{\Wi}-1\right) + m = 1 + \sum_{i=1}^{m-1} \abs{\Wi}.
    \]
    (Interval 3: strict internal) If $p < m < n$, where $t_i \notin S(H)$ for $i = p,\ldots ,m-1$ but $\{t_{p-1},t_m\} \subset S(H)$, then $\{p,\ldots ,m\}$ is associated with the splines
    \[\left\{ \spline{y}_{p,k} \mid k \in [\bn]\right\} \cup \left\{\spline{f}_i^A \mid i=p+1,\ldots ,m-1\right\} \cup \{\spline{r}_m\},\]
    of which there are $\ds 2n + 1 + \sum_{i=p+1}^{m-1}\abs{\Wi}$. On the other hand 
    \begin{align*}
    \sum_{i=p}^m \abs{\hdes{i}} &= 2n-p + m + \sum_{i=p+1}^{m-1} \left(\abs{\Wi}  -1\right)\\&= 2n-p+m - (m-1-p) + \sum_{i=p+1}^{m-1} \abs{\Wi} \\&= 2n+ 1 + \sum_{i=p+1}^{m-1} \abs{\Wi}.
    \end{align*}
    (Interval 4: weak internal) If $p < n-1$, where $t_i \notin S(H)$ for $i = p,\ldots ,n-1$ but $\{t_{p-1},t_n\} \subset S(H)$, then $\{p,\ldots ,n\}$ is associated with the splines
    \[\left\{ \spline{y}_{p,k} \mid k \in [\bn]\right\} \cup \left\{\spline{f}_i^A \mid i=p+1,\ldots ,n-2\right\} \cup \{\spline{r}_{n-1},\spline{h}\} \cup \{\spline{r}_n\},\]
    of which there are $\ds 2n + \sum_{i=p+1}^{n-2}\abs{\Wi} + 3$. On the other hand 
    \begin{align*}
    \sum_{i=p}^n \abs{\hdes{i}} &=(2n-p) + \sum_{i=p+1}^{n-2} (\abs{\Wi}-1) + n+1\\
        &=(2n-p) - (n-2-p) + \sum_{i=p+1}^{n-2} \abs{\Wi} + n+1\\
        %&=n + 2 +  \sum_{i=p+1}^{n-2} \abs{\Wi} + n+1\\
        &=  2n + \sum_{i=p+1}^{n-2}\abs{\Wi} + 3.
    \end{align*}
    (Interval 4: $n-1$ internal point) If $\{t_{n-2},t_{n-1},t_n\} \cap S(H) = \{t_{n-2},t_n\}$ then $\{n-1,n\}$ is associated with the splines
    \[\{\spline{r}_{n-1},\spline{h}\} \cup \{\spline{r}_n\},\]
    of which there are $\ds 3= 2+1 = \abs{\hdes{n-1}} +  \abs{\hdes{n}}$.\\
    (Interval 6: right-open) If $p < n$, $t_{p-1} \in S(H)$,  and $t_i \notin S(H)$ for all $i \geq p$, then $\{p,\ldots ,n\}$ is associated with the splines 
    \[
    \left\{\spline{f}_i^A \mid i=p+1,\ldots ,n\right\} \cup \{\spline{g}_i \mid i \in [n]\},
    \]
    of which there are $\ds n+ \sum_{i=p+1}^{n} \abs{\Wi} $. On the other hand, 
    \[
    \sum_{i=p}^n \abs{\hdes{i}} = 2n-p + \sum_{i=p+1}^{n} \left(\abs{\Wi}  -1\right) = 2n-p - (n-p) + \sum_{i=p+1}^{n} \abs{\Wi} = n + \sum_{i=p+1}^{n} \abs{\Wi}
    \]
    Since every $i \in [n]$ is contained in exactly one type of the intervals above for all $H$, the size of this spanning set matches the $\dim(\splines{H}^1)$, and so it is a basis.
\end{proof}
\begin{lemma}\label{lem:right_basis}
    For each $i$, Let $\left\{A_m \mid m \in \left[2^i\choos{n}{i}\right]\right\}$ be an enumeration of the unbalanced sets of size $i$. Then  
    \begin{align*}
    \spline{T} \cup \spline{R} &\cup \left\{\spline{f}_i^{A_m} - \spline{f}_i^{A_{m+1}} \left| \begin{matrix}\text{ $i$ is uncovered,} \\ m \in \left[2^i\choos{n}{i}-1\right]\end{matrix}\right.\right\} \cup \left\{\spline{y}_{i,k} - \spline{y}_{i,k+1} \left| \begin{matrix} i\text{ is surrounded},\\ \bn \leq k < n\end{matrix}\right.\right\} \\ &\cup \begin{cases}
    \{\spline{g}_i-\spline{g}_{i+1} \mid i \in [n-1] \} &\text{if $t_n \notin S(H)$}\\
    \{\spline{h} \} &\text{if }\{t_{n-1},t_n\} \cap S(H) = \{t_{n}\}\\
    \emptyset &\text{otherwise}
\end{cases}
    \end{align*}
    is a basis for $\splines{H}^1$.
\end{lemma}
\begin{proof}
    This could be done exactly the same as the proof of Lemma \ref{lem:left_basis}, but it is easier to do a change of basis. The relation that gives an $\spline{r}_i$ for $i$ uncovered in return for $\sum_{m} f_i^{A_m}$ is Proposition \ref{prop:relations}(1), and the relation that gives an $\spline{r}_i$ for $i$ surrounded is Proposition \ref{prop:relations}(2). The relations that gives an $\spline{r}_n$ for $t_n \notin S(H)$ is \ref{prop:relations}(5).
\end{proof}
Recall Definition \ref{def:itypes} for the main result below.
\begin{theorem}\label{thm:character}
    Say $\Delta \subsetneq H$. Let $\bm{\chi}$ be the character of the defining representation of $\Wn$ (i.e. the action on $\C^n$). Let $\fh_i$ be the character of the action of $\Wn$ on the cosets of $\mathfrak{S}_i\times \mathfrak{W}_{n-i}$. Let $\fs$ be the character of the action of $\Wn$ on cosets of $\mathfrak{W}_1 \times \mathfrak{W}_{n-1}$. Let $\mathbb{1}$ be the character of the trivial representation and $\bm{\delta}$ be the character $w\mapsto (-1)^{\abs{\Neg(w)}}$. Then 
    \begin{align*}
     \lrep{H}_1 = \abs{\{i \in [n] \mid i\text{ shaded}\,\}}\mathbb{1} &+ \sum_{i\text{ uncovered}} \fh_i \\[6pt]&+ \abs{\{i \in [n-2] \mid i\text{ surrounded}\,\}}\fh_1 \\[6pt]&+ \begin{cases}
    \fs &\text{if $t_n \notin S(H)$}\\
    \bm{\delta} &\text{if }\{t_{n-1},t_n\} \cap S(H) = \{t_{n}\}\\
    0 &\text{otherwise.}
\end{cases}
    \end{align*}
    and 
        \begin{align*}
     \rrep{H}_1 = \bm{\chi} &+ \sum_{i\text{ uncovered}} \left(\fh_i - \mathbb{1}\right) \\[6pt]&+ \abs{\{i \in [n-2] \mid i\text{ surrounded}\,\}}\left(\fh_1-\mathbb{1}\right) \\[6pt]&+ \begin{cases}
    \fs-\mathbb{1} &\text{if $t_n \notin S(H)$}\\
    \bm{\delta} &\text{if }\{t_{n-1},t_n\} \cap S(H) = \{t_{n}\}\\
    0 &\text{otherwise.}
\end{cases}
    \end{align*}
\end{theorem}
\begin{proof}
    Remove $\spline{T}$ from the set in Lemma \ref{lem:left_basis} to get a basis for the left-quotient, and $\spline{R}$ from the set in Lemma \ref{lem:right_basis} to get a basis for the right quotient. From there, recall Lemma \ref{lem:dotact_most} and Lemma \ref{lem:dotact_h} and note that:
    \begin{itemize}[itemsep=6pt]
        \item the character of the action on $\spline{T}$ is $\bm{\chi}$,
        \item the action on $\spline{r}_i$ is trivial,
        \item the character of the action on $\spline{F}_H$ is $\ds  \sum_{i\text{ uncovered}} \fh_i $, 
        \item the character of the action on $\spline{Y}_H$ is $\ds \abs{\{i \in [n-1] \mid i\text{ surrounded}\}}\mathfrak{h_1}$, and that
        \item the character of the action on $\C\{\spline{r}_n-2\,\spline{h}\}$ is $\bm{\delta}$.
    \end{itemize}
    Finally, we have by the relation in Proposition \ref{prop:relations}(5) that $\spline{g}_i + \spline{g}_{\bar{i}} + \spline{t}_i = 2\spline{g}_i$. So as vector spaces $\C\{\spline{t}_i,\spline{g}_i \mid i \in [n]\} = \C\{\spline{t}_i,\spline{g}_i+\spline{g}_{\bar{i}} \mid i \in [n]\}$, which is a decomposition in to two $\Wn$-invariant subspaces. The action of $\Wn$ on $\{\spline{g}_i + \spline{g}_{\bar{i}} \mid i \in [n]\}$ is the same as the permutation action $i \mapsto \abs{w(i)}$, which is precisely the action on cosets of $\mathfrak{W}_1 \times \mathfrak{W}_{n-1}$, whose character is  is $\fs$. This is sufficient for computing the left character.\par 
    
    The right character follows from the general procedure of ``removing" from any permutation representation with permutation basis $\{v_i\}$ the trivial sub-representation $\C\left\{ \sum v_i \right\}$. In particular, if $\mathfrak{p}$ is the character of a representation with permutation basis $\{v_i\}$, then the character on the complementary subspace $\C\{v_i-v_{i+1}\}$ is $\mathfrak{p}-\mathbb{1}$.
\end{proof}
In type B, $t_n \in S(H)$ implies that $t_{n-1} \in S(H)$, and so it is impossible to get the $\bm{\delta}$ term from Theorem \ref{thm:character}. In type C the opposite is true: $t_{n-1} \in S(H)$ implies that $t_n \in S(H)$, and so $\delta$ may appear. This is the ``moreover" part of Theorem \ref{intthm:character}.
\begin{example}
    The table below gives every $\lrep{H}_1$ and $\lrep{H}_1$ for $n=4$.
\begin{table}[H]
\centering
\renewcommand{\arraystretch}{1.5}
\begin{tabular}{|c|p{2in}|p{2in}|p{0.3in}|}\hline
    { $t_i \in S(H)$} & {\large $\lrep{H}_1$} & \large $\rrep{H}_1$ &\large dim\\\hline
    $\emptyset$ & $\fh_1 +\fh_2+\fh_3+\fh_4-\bchi$  & $\fh_1 +\fh_2+\fh_3+\fh_4-4\triv$ & $76$ \\\hline
    $\{t_1\}$ & $\triv + \fh_3 + \fh_4 + \fs$   &$\bchi + \fh_3 + \fh_4 + \fs -3\triv$ & $53$\\\hline
    $\{t_2\}$ & $\triv+\fh_1+\fh_4+\fs $        &$\bchi+\fh_1+\fh_4+\fs -3\triv$ &$29$\\\hline
    $\{t_3\}$ & $\triv+\fh_1+\fh_2+\fs $        &$\bchi+\fh_1+\fh_2+\fs -3\triv$ &$37$\\\hline
    $\{t_4\}$ & $2\triv+\fh_1+\fh_2+\bm{\delta}$ &$\bchi+\fh_1+\fh_2+\bm{\delta}-2\triv$ &$35$\\\hline
    $\{t_1,t_2\}$ & $2\triv+\fh_4+\fs $         &$\bchi+\fh_4+\fs -2\triv$ &$22$\\\hline
    $\{t_1,t_3\}$ & $2\triv+\fh_1+\fs $         &$\bchi+\fh_1+\fs -2\triv$ &$14$\\\hline
    $\{t_1,t_4\}$ & $3\triv+\fh_1 +\bm{\delta}$ &$\bchi+\fh_1 +\bm{\delta}-\triv$ &$12$\\\hline
    $\{t_2,t_3\}$ & $2\triv+\fh_1+\fs $         &$\bchi+\fh_1+\fs -2\triv$ &$14$\\\hline
    $\{t_2,t_4\}$ & $3\triv+\fh_1+\bm{\delta} $ &$\bchi+\fh_1+\bm{\delta} -\triv$ &$12$\\\hline
    $\{t_3,t_4\}$ & $2\triv+\fh_1+\fh_2 $       &$\bchi+\fh_1+\fh_2 -2\triv$ &$34$\\\hline
    $\{t_1,t_2,t_3\}$ &$3\triv+\fs $            &$\bchi+\fs -\triv$ &$7$\\\hline
    $\{t_1,t_2,t_4\}$ &$4\triv+\bm{\delta} $    &$\bchi+\bm{\delta} $ &$5$\\\hline
    $\{t_1,t_3,t_4\}$ &$3\triv+\fh_1$           &$\bchi+\fh_1-\triv$ &$11$\\\hline
    $\{t_2,t_3,t_4\}$ & $3\triv+\fh_1 $         &$\bchi+\fh_1 -\triv$ &$11$\\\hline
    $\{t_1,t_2,t_3,t_4\}$ & $4\mathbb{1}$       &$\bchi$ &$4$\\\hline
    
\end{tabular}\vspace*{-2pt}
\captionsetup{width=6.2in}
\caption{All $\lrep{H}_1$ and $\rrep{H}_1$ for $n=4$.}
\label{tab:W4_repns}
\end{table}\vspace{-10pt}

\end{example}
\begin{example}
    Theorem \ref{thm:character} is easily applied to large $n$. If $n = 8$ and $H$ is such that $S(H) \cap \{t_i \mid i \in [n]\} = \{t_2,t_5,t_6,t_8\}$,
    then $\{1,4\}$ are uncovered, $\{2,5,6,7,8\}$ are shaded, $3$ is surrounded and $\{t_7,t_8\} \cap S(H) = \{t_8\}$. So 
    \[
    \lrep{H}_1 = 5\mathbb{1} + \fh_1+\fh_4 + \fh_1 + \bm{\delta}
    \]
    and 
    \[
    \rrep{H}_1 = \bchi + \fh_1+\fh_4 + \fh_1 + \bm{\delta} -3\triv, 
    \]
    both of which are representations of dimension $1158$.
\end{example}

\appendix
\section{Type B/C-Symmetric Functions.}\label{sec:symmetricfunctions}
Symmetric functions (with coefficients in $\C$) are formal power series in $\C[x_1,\ldots]$ invariant under any permutation of the variables. The symmetric functions form a graded ring over $\C$ denoted by $\Lambda(x) = \bigoplus_{n \geq 0} \Lambda_n(x)$ with several important bases. Each of these bases is indexed by partitions of positive integers. The bases employed here are the power sum, complete homogeneous, and Schur symmetric functions, denoted by $\{p_\lambda(x)\}$, $\{h_\lambda(x)\}$, and $\{s_\lambda(x)\}$, respectively.\par 
The \emph{type B/C symmetric functions} are formal power series in $\C[x_1,y_1,\ldots]$ invariant under any permutation of the $\{x_1,\ldots\}$ and $\{y_1,\ldots\}$, respectively. The type B/C symmetric functions also form a graded ring, denoted $\Lambda(x,y)$, where each graded piece is given by 
\[
\Lambda_n(x,y) \coloneqq \bigoplus_{k=0}^n \Lambda_k(x) \otimes \Lambda_{n-k}(y)
\]
The type B/C symmetric functions have several important bases indexed by pairs of partitions $\lambda,\mu$. Employed here are the plythestic power sum, double-homogeneous, and double-Schur bases, denoted by $\{p_\lambda(x\!+\!y),p_\mu(x\!-\!y)\}$, $\{h_\lambda(x)h_\mu(y)\}$, and $\{s_\lambda(x)s_\mu(y)\}$, respectively. To shorten notation, particularly in Table \ref{tab:reps}, we may denote $h_\lambda(x)h_\mu(y)$ and $s_\lambda(x)s_\lambda(y)$ as $h_{\lambda,\mu}$ and $s_{\lambda,\mu}$ respectively.\par 
The ring $\Lambda(x)$ is interesting for a great variety of reasons, one of which is that the \emph{Frobenius characteristic map} gives an isometric isomorphism from the character space of $\Sn$ to $\Lambda_n(x)$. This map is defined via the power-sum basis. If $\chi$ is a character of $\Sn$, then 
\[
\mathrm{Frob}(\chi) := \frac{1}{n!}\sum_{w \in \Sn} \chi(w) p_{\lambda(w)}(x), \;\;\text{where}\;\; \lambda(w) \text{ is the cycle type of } w.
\] 
This map sends the characters of irreducible representations to the Schur basis and the characters of permutation representations on cosets of Young subgroups to the complete homogeneous basis. \par 
Similarly, the type B/C-Frobenius characteristic map gives an isometric isomorphism from the character space of $\Wn$ to $\Lambda_n(x,y)$ \cite[\S1 Appendix B]{Macdonald_symfuncs}. If $\chi$ is a character of $\Wn$, then
\[
\mathrm{Frob}_{BC}(\chi) := \frac{1}{2^nn!}\sum_{w \in \Wn} \chi(w) p_{\lambda(w)}(x\!+\!y)p_{\mu(w)}(x\!-\!y),
\] 
where $\lambda(w),\mu(w)$ is the signed-cycle type of $w$. Even for common characters and small $n$ it can be difficult to compute $\mathrm{Frob}_{BC}(\chi)$ in other bases from the definition.\par 
\begin{example}\label{ex:BCFrob}
    The character of the defining representation $\bm{\chi}$ on $\mathfrak{W}_3$ is 
    \[
    \bm{\chi}(w) = \abs{\{i \in [n] \mid w(i) = i\}} - \abs{\{i \in [n] \mid w(i) = -i\}}.
    \]
    The following are all permutations in $\mathfrak{W}_3$ that are supported on $\bm{\chi}$:
    \begin{align*}
        \{w \in \mathfrak{W}_3 \mid \bm{\chi}(w) = -3\} &= \{[\bar{1},\bar{2},\bar{3}]\} \\
        \{w \in \mathfrak{W}_3 \mid \bm{\chi}(w) = -1\} &= \{
                    [\bar{1},\pm 3, \pm 2], [\pm 3,\bar{2},\pm 1], [\pm 2,\pm 1,\bar{3}],
                    [1,\bar{2},\bar{3}], [\bar{1},2,\bar{3}], [\bar{1},\bar{2},3]
                    \} \\
        \{w \in \mathfrak{W}_3 \mid \bm{\chi}(w) = 1\} &= \{
                    [1,\pm 3, \pm 2], [\pm 3,2,\pm 1], [\pm 2,\pm 1,3],
                    [1,2,\bar{3}], [1,\bar{2},3], [\bar{1},2,3]
                    \} \\
        \{w \in \mathfrak{W}_3 \mid \bm{\chi}(w) = 3\} &= \{[1,2,3]\}.
    \end{align*}
    So there are $32$ summands in $\Frob_{BC}(\bm{\chi})$ for $n = 3$. From there, the signed cycle-types are  as follows:
    \begin{center} \renewcommand{\arraystretch}{1.5}
    \begin{tabular}{|r|l|}\hline 
        $\lambda(w),\mu(w)$ & Supported $w \in \mathfrak{W}_3$  \\\hline
        $(2,1),\emptyset$ & $[1,3,2],[3,2,1],[2,1,3],[1,\bar{3},\bar{2}],[\bar{3},2,\bar{1}],[\bar{2},\bar{1},3]$ \\\hline
        $(2),(1)$ & $[\bar{1},3,2],[3,\bar{2},1],[2,1,\bar{3}],[\bar{1},\bar{3},\bar{2}],[\bar{3},\bar{2},\bar{1}],[\bar{2},\bar{1},\bar{3}]$ \\\hline
        $(1),(2)$ & $[1,3,\bar{2}],[3,2,\bar{1}],[2,\bar{1},3],[1,\bar{3},2],[\bar{3},2,1],[\bar{2},1,3]$ \\\hline
        $\emptyset,(2,1)$ & $[\bar{1},3,\bar{2}],[3,\bar{2},\bar{1}],[2,\bar{1},\bar{3}],[1,\bar{3},\bar{2}],[\bar{3},\bar{2},1],[2,\bar{1},\bar{3}]$ \\\hline
        $(1,1),(1)$ & $[1,2,\bar{3}], [1,\bar{2},3], [\bar{1},2,3]$  \\\hline
        $(1),(1,1)$ & $[1,\bar{2},\bar{3}], [\bar{1},2,\bar{3}], [\bar{1},\bar{2},3]$  \\\hline
        $(1,1,1),\emptyset$ & $[1,2,3]$ \\\hline
        $\emptyset,(1,1,1)$ & $[\bar{1},\bar{2},\bar{3}]$ \\\hline
    \end{tabular}
    \end{center}
    Even still, one must transform to the power-sum basis via $2p_1(x) = p_1(x+y)+p_1(x-y)$ and $2p_1(y) = p_1(x+y)-p_1(x-y)$ (which is sufficient here only because $n$ is small), then from the plethystic to homogeneous basis. One finds that $\Frob_{BC}(\bm{\chi}) = h_{(2),(1)}$ \cite{BCFromEx}.
\end{example}
It can often be much easier to compute $\Frob_{BC}$ using the rules for induced representations Table \ref{tab:reps}.
\begin{proposition}
    The defining character $\bm{\chi}$ on $\Wn$ is isomorphic to the induced character $\left(\bm{\delta} \times \mathbb{1} \right)_{\mathfrak{W}_1 \times \mathfrak{W}_{n-1}}^{\Wn}$. Moreover, $\Frob_{BC}(\bm{\chi}) = h_{(n-1),(1)}$.
\end{proposition}
\begin{proof}
    It follows from the definition that 
    \[
    \bm{\chi}(w) = \abs{\{i \in [n] \mid w(i) = i\}} - \abs{\{i \in [n] \mid w(i) = -i\}}.
    \]
    Now $\mathfrak{W}_1 \times \mathfrak{W}_{n-1}$ are those $w \in \Wn$ such that $w(1) = \pm 1$, and the character evaluation $\bm{\delta} \times \mathbb{1} (w)$ is $1$ if $w(1) = 1$ and $-1$ if $w(1) = -1$. Conjugation in $\Wn$ sends fixed points to fixed points and negative fixed points to negative fixed points. Formally, 
    \[\{i \in [n] \mid v^{-1}wv(i) = i\} = \left\{\abs{v^{-1}(i)} \mid w(i) = i\right\}\]
    and 
    \[\{i \in [n] \mid v^{-1}wv(i) = \bar{i}\} = \left\{\abs{v^{-1}(i)} \mid w(i) = \bar{i}\right\}.\]
    Furthermore, the cosets $v(\mathfrak{W}_1 \times \mathfrak{W}_{n-1})$ are entirely determined by the value $v(1)$. Finally, we can let $H \coloneqq \mathfrak{W}_1 \times \mathfrak{W}_{n-1}$ for notation, and compute directly
    \begin{align*}
        \left(\bm{\delta} \times \mathbb{1} \right)_{H}^{\Wn}(w) &= \sum_{\substack{vH \\ 
        v^{-1}wv \in H}} \bm{\delta} \times \mathbb{1}(v^{-1}wv) \\
        &= \sum_{\substack{vH \\ 
        v^{-1}wv \in H \\ v(1) > 0}} \bm{\delta} \times \mathbb{1}(v^{-1}wv) + \sum_{\substack{vH \\ 
        v^{-1}wv \in H \\ v(1) < 0}} \bm{\delta} \times \mathbb{1}(v^{-1}wv) \\
        &= \abs{\{i \in [n] \mid v^{-1}wv(i) = i\}}- \abs{\{i \in [n] \mid v^{-1}wv(i) = \bar{i}\}} \\
        &= \bm{\chi}(w).
    \end{align*}
\end{proof}

 The irreducible representations of $\Wn$, as indexed by pairs of partitions, can be constructed in several different ways, as the wreath product of $\Sn$ with $\Z_2$ (see \cite[\S2]{GeissKinch_hyperoctChars}, \cite[\S1,Appendix B]{Macdonald_symfuncs} or \cite[\S7]{Zelevinsky_reps_classicalgroups}), and as the Weyl group of Lie type B/C (see \cite[\S5]{NilpotentOrbits}). Under the Frobenius map, the character $\chi^{(\lambda,\mu)}$ of an irreducible representation is mapped to the double-Schur basis element $s_\lambda(x)s_\mu(y)$. \par 

     The double-Schur basis is related to the double-homogeneous basis in $\Lambda(x,y)$ in precisely the same way that the Schur basis is related to the homogeneous basis $\Lambda(x)$. In particular, $h_\lambda(x) = \sum_\gamma K_{\gamma,\lambda} s_\gamma(x)$ where $K_{\gamma,\lambda}$ is a \emph{Kostka number}, and so 
     \[\ds h_{\lambda}(x)h_\mu(y) = \sum_{\gamma}\sum_{\nu} K_{\gamma,\lambda}K_{\nu,\mu} s_\gamma(x)s_\nu(y).\] 

In general, $h_\lambda(x)h_\mu(y)$ does \textbf{not} correspond to an action on cosets of Young subgroups. In fact, those representations do not span the representation space of $\Wn$. This particular issue has raised questions as to what is an appropriate type B/C analogue of the graded Stanley-Stembridge conjecture of Shareshian and Wachs.
\begin{conjecture}[The Graded Stanley-Stembridge Conjecture \cite{StanleyStembridge,SW2016chromaticquasisymmetric}]\label{conj:stanstem}
    If $H$ is a type A Hessenberg space, then $\Frob(\lrep{H})$ is $h_\lambda(x)$-positive.
\end{conjecture}
There have been many partial results in service of proving Conjecture A.1 \cite{STANLEY1995chromsym, Gasharov96, GuayPaquet13, SW2016chromaticquasisymmetric,  harada2017cohomology, brosnan_chow_dotactn_is_chromsym,Dahlberg19, Abreu_Nigro20}, and a proof of the ungraded version was recently given by Hikita \cite{Hikita_stanstem}.\par 
In representation-theoretic terms, Conjecture \ref{conj:stanstem} asks that the character $\lrep{H}$ be written as a non-negative sum of characters of the action of $\Sn$ on its Young subgroups. As mentioned, this is not possible for $\Wn$, and so the proper generalization (if any) is up for debate. 
\begin{question}\label{que:hpos}
    Does each graded piece of $\Frob_{BC}(\lrep{H})$ expand with non-negative coefficients in the $h_\lambda(x)h_\mu(y)$-basis of $\Lambda(x,y)$?
\end{question}
Comparing Theorem \ref{thm:character} and Table \ref{tab:reps}, this paper has answered Question \ref{que:hpos} in the affirmative for the degree-one graded piece. A result in the ungraded case would also be quite interesting. \par 
In type A (so $\Sn$), every reflection subgroup is conjugate to a Young subgroup $\mathfrak{S}_\lambda$ for partitions $\lambda$, and the representation ring is freely generated by the permutation actions on the cosets of these subgroups. This is not the case in other types, in particular types B and C, where every Young subgroup is of type $\mathfrak{S}_\lambda \times \mathfrak{W}_{n-\abs{\lambda}}$. However, it \textbf{is} possible to generate the entire representation ring if one allows the permutation actions on other types of reflection subgroups. In particular, it suffices to consider the subgroups of type $\mathfrak{S}_\lambda \times \mathfrak{W}_\mu$ for pairs of partitions $\lambda,\mu$. The associated representations freely generate the representation ring (see the proof of Theorem 4.3 in \cite{Stembridge_Permuto}). So we can ask the following, even stronger question.
\begin{question}\label{que:refpos}
    Can the character $\lrep{H}$ be written as a non-negative sum of characters of the action of $\Wn$ on its reflection subgroups?
\end{question}
This paper has also answered Question \ref{que:refpos} in the affirmative for the degree-one graded piece, and any result in the ungraded case would also be quite interesting. We note that it may be necessary to consider reflection subgroups of type D. \par 
Conjecture \ref{conj:stanstem} was originally formulated as the $e$-positivity of \emph{chromatic quasisymmetric functions} of particular types of graphs \cite{StanleyStembridge,SW2016chromaticquasisymmetric}. The connection between these constructions was conjectured by Sharesian and Wachs, and later proven by Brosnan and Chow via monodromy and Guay-Paquet via Hopf algebras \cite{SW2016chromaticquasisymmetric,brosnan_chow_dotactn_is_chromsym,guaypaquet2016shar_wachs_conj}. This connection begs the question of whether there is a similar one in types B and C. Clearly, one would have to consider coloring schemes with two families of colors in order to get a type B/C symmetric function in two variables. It is unclear whether any of the currently-studied chromatic symmetric function generalizations \cite{Harary_balance_signedgraph,Zaslavsky_signedgraphs,Reiner_signedposets,Csar_rootrings,KurodaTsujie_signedchromsym,Skandera_hyperoct} are equal to $\Frob_{BC}(\lrep{H})$ in types B or C. 
\begin{question}
    Is there a reasonable graph-coloring schema for graphs built from $H$ such that the corresponding two-variable chromatic symmetric function is $\Frob_{BC}(\lrep{H})$?
\end{question}

For the right character $\rrep{H}$, in type A there is the major open problem of finding a combinatorial formula for the coefficients of $s_\lambda(x)$ in $\Frob(\rrep{H})$, a question motivated by the study of unicellular LLT polynomials \cite{Leclerc2000LLT_LR_KL,Blasiak2016LLTMacdonald,guaypaquet2016shar_wachs_conj,Ayzenberg2018isospectral, ALEXANDERSSON2018LLTchromsym,Huh2020LLTLollipop,Alexandersson2020Lollipop,Lee2021LLTlinear}. Here, the appropriate generalization is much clearer.
\begin{question}\label{que:rightComb}
    Can one find any formula, in particular a combinatorial formula, for the coefficients of $s_\lambda(x)s_\mu(y)$ in $\Frob_{BC}(\rrep{H})$? 
\end{question}
Again, comparing Theorem \ref{thm:character} and Table \ref{tab:reps}, this paper has answered Question \ref{que:rightComb} in the affirmative for the degree-one graded piece.

\clearpage 
\begin{table}[H]
    \centering
    \renewcommand{\arraystretch}{2.5}
    \begin{tabular}{|p{3.5in}|p{2.3in}|}\hline
        {\large Representations or  Characters of $\Wn$} & {\large Expansions in $\Lambda_n(x,y)$} \\\hline
        The trivial character $\mathbb{1}$ & $s_{(n),\emptyset} = h_{(n),\emptyset}$ \\\hline
        The sign character $\mathrm{sgn}$ & $s_{\emptyset,(1^n)}$ \\\hline
        $\bm{\delta}$ where $\bm{\delta}(w) = (-1)^{\abs{\Neg(w)}}$ & $s_{\emptyset,(n)} = h_{\emptyset,(n)}$ \\\hline
        The irreducible character $\chi^{(\lambda,\mu)}$ & $s_{\lambda,\mu}$ \\\hline
        $\mathrm{sgn} \otimes \chi^{(\lambda,\mu)} $ & $s_{\mu^t,\lambda^t}$ \\\hline
        $\bm{\delta} \otimes \chi^{(\lambda,\mu)} $ &  $s_{\mu,\lambda}$ \\\hline
        The defining representation on $\C^n$ (so $\bm{\chi}$) & $h_{(n-1),(1)}$ \\\hline 
        Action on cosets of $\mathfrak{S}_k\times \mathfrak{W}_{n-k}$ (so $\fh_k$) & $\ds h_{(n-k),\emptyset} \sum_{j=0}^{\lambda_i} h_{(j),(k-j)}$ \\\hline
        Action on cosets of $\mathfrak{W}_1\times \mathfrak{W}_{n-1}$ (so $\fs$) & $\ds h_{(n-1,1),\emptyset}$ \\\hline
        Action on cosets of $\mathfrak{S}_\lambda \times \mathfrak{W}_\mu$ & $\ds h_{\mu,\emptyset} \prod_{\lambda_i \in \lambda} \sum_{j=0}^{\lambda_i} h_{(j),(\lambda_i-j)}$ \\\hline
        Action on cosets of the type D${}_n$ subgroup & $\ds h_{(n),\emptyset} + h_{\emptyset,(n)} $ \\\hline
        An induced character $\left(\chi \times \theta\right)_{\mathfrak{W}_k \times \mathfrak{W}_\ell}^{\Wn}$ & $\Frob_{BC}(\chi)\Frob_{BC}(\theta)$\\\hline
        The induced character $\left(\chi^{(\lambda,\mu)} \times \chi^{(\alpha,\beta)}\right)_{\mathfrak{W}_k \times \mathfrak{W}_\ell}^{\Wn}$ & $(s_\lambda(x)\cdot s_\alpha(x))(s_\mu(y)\cdot s_\beta(y))$\\\hline
        The induced character $\left(\mathbb{1}\right)_{\Sn}^{\Wn}$ & ${\ds \sum_{k=0}^n s_{(k),(n-k)}  = \sum_{k=0}^n h_{(k),(n-k)}}$ \\\hline
        The induced character $\left(\mathbb{1} \times \bm{\delta} \right)_{\mathfrak{W}_\lambda \times \mathfrak{W}_\mu}^{\Wn}$ & $\ds h_{\lambda,\mu} = \sum_{\gamma}\sum_{\nu} K_{\gamma,\lambda}K_{\nu,\mu}s_\gamma(x)s_\nu(y)$ \\\hline
    \end{tabular}\vspace*{4pt}
    \captionsetup{width=6.2in}
    \caption{Representations of $\Wn$ and  Frobenius characteristics \cite{GeissKinch_hyperoctChars,Macdonald_symfuncs, Zelevinsky_reps_classicalgroups, Stembridge_Permuto, Stembridge_guide, Skandera_hyperoct}.}
    \label{tab:reps}
\end{table}
%\vspace*{12pt} 
\clearpage 

\renewcommand{\baselinestretch}{0.97} 

\printbibliography

@article {Lesnev25,
    AUTHOR = {Lesnevich, Nathan R. T.},
     TITLE = {Splines on {C}ayley graphs of the symmetric group},
   JOURNAL = {Forum Math. Sigma},
  FJOURNAL = {Forum of Mathematics. Sigma},
    VOLUME = {13},
      YEAR = {2025},
     PAGES = {Paper No. e96},
      ISSN = {2050-5094},
   MRCLASS = {05E10 (05C25 05E05 20C30)},
  MRNUMBER = {4922765},
       DOI = {10.1017/fms.2025.10037},
       URL = {https://doi.org/10.1017/fms.2025.10037},
}

@book {HumphreysReflections,
    AUTHOR = {Humphreys, James E.},
     TITLE = {Reflection groups and {C}oxeter groups},
    SERIES = {Cambridge Studies in Advanced Mathematics},
    VOLUME = {29},
 PUBLISHER = {Cambridge University Press, Cambridge},
      YEAR = {1990},
     PAGES = {xii+204},
      ISBN = {0-521-37510-X},
   MRCLASS = {20-02 (20F32 20F55 20G15 20H15)},
  MRNUMBER = {1066460},
MRREVIEWER = {Louis\ Solomon},
       DOI = {10.1017/CBO9780511623646},
       URL = {https://doi.org/10.1017/CBO9780511623646},
}

@article {PrecupBCPaving,
    AUTHOR = {Precup, Martha},
     TITLE = {Affine pavings of {H}essenberg varieties for semisimple
              groups},
   JOURNAL = {Selecta Math. (N.S.)},
  FJOURNAL = {Selecta Mathematica. New Series},
    VOLUME = {19},
      YEAR = {2013},
    NUMBER = {4},
     PAGES = {903--922},
      ISSN = {1022-1824,1420-9020},
   MRCLASS = {14L35 (14F25 14M15 17B08)},
  MRNUMBER = {3131491},
MRREVIEWER = {Nicolas\ Perrin},
       DOI = {10.1007/s00029-012-0109-z},
       URL = {https://doi.org/10.1007/s00029-012-0109-z},
}

@book {BjornerBrentiCoxeter,
    AUTHOR = {Bj\"orner, Anders and Brenti, Francesco},
     TITLE = {Combinatorics of {C}oxeter groups},
    SERIES = {Graduate Texts in Mathematics},
    VOLUME = {231},
 PUBLISHER = {Springer, New York},
      YEAR = {2005},
     PAGES = {xiv+363},
      ISBN = {978-3540-442387},
   MRCLASS = {05-01 (05E15 20F55)},
  MRNUMBER = {2133266},
MRREVIEWER = {Jian-yi\ Shi},
}

@article {Stembridge_Permuto,
    AUTHOR = {Stembridge, John R.},
     TITLE = {Some permutation representations of {W}eyl groups associated
              with the cohomology of toric varieties},
   JOURNAL = {Adv. Math.},
  FJOURNAL = {Advances in Mathematics},
    VOLUME = {106},
      YEAR = {1994},
    NUMBER = {2},
     PAGES = {244--301},
      ISSN = {0001-8708,1090-2082},
   MRCLASS = {20C15 (13F50 14M25 20F55 20H15)},
  MRNUMBER = {1279220},
MRREVIEWER = {Robert\ B.\ Howlett},
       DOI = {10.1006/aima.1994.1058},
       URL = {https://doi.org/10.1006/aima.1994.1058},
}

@book {Macdonald_symfuncs,
    AUTHOR = {Macdonald, I. G.},
     TITLE = {Symmetric functions and {H}all polynomials},
    SERIES = {Oxford Mathematical Monographs},
   EDITION = {Second},
      NOTE = {With contributions by A. Zelevinsky,
              Oxford Science Publications},
 PUBLISHER = {The Clarendon Press, Oxford University Press, New York},
      YEAR = {1995},
     PAGES = {x+475},
      ISBN = {0-19-853489-2},
   MRCLASS = {05E05 (05-02 20C30 20C33 20K01 33C80 33D80)},
  MRNUMBER = {1354144},
MRREVIEWER = {John\ R.\ Stembridge},
}

@book {Zelevinsky_reps_classicalgroups,
    AUTHOR = {Zelevinsky, Andrey V.},
     TITLE = {Representations of finite classical groups},
    SERIES = {Lecture Notes in Mathematics},
    VOLUME = {869},
      NOTE = {A Hopf algebra approach},
 PUBLISHER = {Springer-Verlag, Berlin-New York},
      YEAR = {1981},
     PAGES = {iv+184},
      ISBN = {3-540-10824-6},
   MRCLASS = {20C30 (16A24 20G05)},
  MRNUMBER = {643482},
MRREVIEWER = {A.\ Kh.\ Kushkule\u i},
}

@article {GeissKinch_hyperoctChars,
    AUTHOR = {Geissinger, L. and Kinch, D.},
     TITLE = {Representations of the hyperoctahedral groups},
   JOURNAL = {J. Algebra},
  FJOURNAL = {Journal of Algebra},
    VOLUME = {53},
      YEAR = {1978},
    NUMBER = {1},
     PAGES = {1--20},
      ISSN = {0021-8693},
   MRCLASS = {20C15},
  MRNUMBER = {491917},
MRREVIEWER = {P.\ M.\ van den Broek},
       DOI = {10.1016/0021-8693(78)90200-4},
       URL = {https://doi.org/10.1016/0021-8693(78)90200-4},
}

@book {NilpotentOrbits,
    AUTHOR = {Collingwood, David H. and McGovern, William M.},
     TITLE = {Nilpotent orbits in semisimple {L}ie algebras},
    SERIES = {Van Nostrand Reinhold Mathematics Series},
 PUBLISHER = {Van Nostrand Reinhold Co., New York},
      YEAR = {1993},
     PAGES = {xiv+186},
      ISBN = {0-534-18834-6},
   MRCLASS = {17-02 (17B20 17B25 22E60)},
  MRNUMBER = {1251060},
MRREVIEWER = {Stephen\ Slebarski},
}

@misc{Stembridge_guide,
    AUTHOR = {Stembridge, John R.},
     TITLE = {A Guide to Working with {W}eyl Group Representations, With Special Emphasis on Branching Rules},
   JOURNAL = {Atlas of Lie Groups AIM Workshop IV},
      YEAR = {2006},
       URL = {http://www.liegroups.org/papers/summer06/what.pdf},
}

@misc{Skandera_hyperoct,
      title={Hyperoctahedral group characters and a type-BC analog of graph coloring}, 
      author={Mark Skandera},
      year={2024},
      eprint={2402.04148},
      archivePrefix={arXiv},
      primaryClass={math.CO},
      url={https://arxiv.org/abs/2402.04148}, 
}

@article {StanleyStembridge,
    AUTHOR = {Stanley, Richard P. and Stembridge, John R.},
     TITLE = {On immanants of {J}acobi-{T}rudi matrices and permutations
              with restricted position},
   JOURNAL = {J. Combin. Theory Ser. A},
  FJOURNAL = {Journal of Combinatorial Theory. Series A},
    VOLUME = {62},
      YEAR = {1993},
    NUMBER = {2},
     PAGES = {261--279},
      ISSN = {0097-3165,1096-0899},
   MRCLASS = {05E15 (20C30)},
  MRNUMBER = {1207737},
MRREVIEWER = {A.\ Kerber},
       DOI = {10.1016/0097-3165(93)90048-D},
       URL = {https://doi.org/10.1016/0097-3165(93)90048-D},
}

@misc{Hikita_stanstem,
      title={A proof of the Stanley-Stembridge conjecture}, 
      author={Tatsuyuki Hikita},
      year={2024},
      eprint={2410.12758},
      archivePrefix={arXiv},
      primaryClass={math.CO},
      url={https://arxiv.org/abs/2410.12758}, 
}

@phdthesis{Csar_rootrings,
    AUTHOR = {Csar, Sebastian Alexander},
     TITLE = {Root and weight semigroup rings for signed posets},
      NOTE = {Thesis (Ph.D.)--University of Minnesota},
    SCHOOL = {University of Minnesota},
 PUBLISHER = {ProQuest LLC, Ann Arbor, MI},
      YEAR = {2014},
     PAGES = {183},
      ISBN = {978-1321-31336-9},
   MRCLASS = {99-05},
  MRNUMBER = {3301700},
       URL =     {http://gateway.proquest.com/openurl?url_ver=Z39.88-2004&rft_val_fmt=info:ofi/fmt:kev:mtx:dissertation&res_dat=xri:pqm&rft_dat=xri:pqdiss:3643596},
}

@article {Harary_balance_signedgraph,
    AUTHOR = {Harary, Frank},
     TITLE = {On the notion of balance of a signed graph},
   JOURNAL = {Michigan Math. J.},
  FJOURNAL = {Michigan Mathematical Journal},
    VOLUME = {2},
      YEAR = {1954},
     PAGES = {143--146},
      ISSN = {0026-2285,1945-2365},
   MRCLASS = {56.0X},
  MRNUMBER = {67468},
MRREVIEWER = {G.\ A.\ Dirac},
       URL = {http://projecteuclid.org/euclid.mmj/1028989917},
}

@misc{KurodaTsujie_signedchromsym,
      title={Chromatic Signed-Symmetric Functions of Signed Graphs}, 
      author={Masamichi Kuroda and Shuhei Tsujie},
      year={2021},
      eprint={2101.03018},
      archivePrefix={arXiv},
      primaryClass={math.CO},
      url={https://arxiv.org/abs/2101.03018}, 
}

@article {Reiner_signedposets,
    AUTHOR = {Reiner, Victor},
     TITLE = {Signed posets},
   JOURNAL = {J. Combin. Theory Ser. A},
  FJOURNAL = {Journal of Combinatorial Theory. Series A},
    VOLUME = {62},
      YEAR = {1993},
    NUMBER = {2},
     PAGES = {324--360},
      ISSN = {0097-3165,1096-0899},
   MRCLASS = {06A07 (05E99)},
  MRNUMBER = {1207741},
MRREVIEWER = {Sergey\ V.\ Fomin},
       DOI = {10.1016/0097-3165(93)90052-A},
       URL = {https://doi.org/10.1016/0097-3165(93)90052-A},
}

@article {Zaslavsky_signedgraphs,
    AUTHOR = {Zaslavsky, Thomas},
     TITLE = {Signed graphs},
   JOURNAL = {Discrete Appl. Math.},
  FJOURNAL = {Discrete Applied Mathematics. The Journal of Combinatorial
              Algorithms, Informatics and Computational Sciences},
    VOLUME = {4},
      YEAR = {1982},
    NUMBER = {1},
     PAGES = {47--74},
      ISSN = {0166-218X,1872-6771},
   MRCLASS = {05C99 (05B35)},
  MRNUMBER = {676405},
MRREVIEWER = {F.\ Harary},
       DOI = {10.1016/0166-218X(82)90033-6},
       URL = {https://doi.org/10.1016/0166-218X(82)90033-6},
}

@article {BalibanuCrooks_sheaves,
    AUTHOR = {B\u alibanu, Ana and Crooks, Peter},
     TITLE = {Perverse sheaves and the cohomology of regular {H}essenberg
              varieties},
   JOURNAL = {Transform. Groups},
  FJOURNAL = {Transformation Groups},
    VOLUME = {29},
      YEAR = {2024},
    NUMBER = {3},
     PAGES = {909--933},
      ISSN = {1083-4362,1531-586X},
   MRCLASS = {14M17 (14M15 20B27)},
  MRNUMBER = {4788018},
       DOI = {10.1007/s00031-022-09755-3},
       URL = {https://doi.org/10.1007/s00031-022-09755-3},
}

@article {PrecupSommers_sheaves,
    AUTHOR = {Precup, Martha and Sommers, Eric},
     TITLE = {Perverse sheaves, nilpotent {H}essenberg varieties, and the
              modular law},
   JOURNAL = {Pure Appl. Math. Q.},
  FJOURNAL = {Pure and Applied Mathematics Quarterly},
    VOLUME = {21},
      YEAR = {2025},
    NUMBER = {1},
     PAGES = {495--540},
      ISSN = {1558-8599,1558-8602},
   MRCLASS = {14M15 (17B08)},
  MRNUMBER = {4847243},
       DOI = {10.4310/pamq.241203042708},
       URL = {https://doi.org/10.4310/pamq.241203042708},
}

@misc{BCFromEx,
  author = "Liao, Hsin-Chieh",
  date = "2025",
  month = "10",
  day = "12",
  howpublished = "personal communication"
  }

@article {Blasiak2016LLTMacdonald,
    AUTHOR = {Blasiak, Jonah},
     TITLE = {Haglund's conjecture on 3-column {M}acdonald polynomials},
   JOURNAL = {Math. Z.},
  FJOURNAL = {Mathematische Zeitschrift},
    VOLUME = {283},
      YEAR = {2016},
    NUMBER = {1-2},
     PAGES = {601--628},
      ISSN = {0025-5874,1432-1823},
   MRCLASS = {05E05 (33C52)},
  MRNUMBER = {3489082},
MRREVIEWER = {Meesue\ Yoo},
       DOI = {10.1007/s00209-015-1612-7},
       URL = {https://doi.org/10.1007/s00209-015-1612-7},
}

@incollection {Leclerc2000LLT_LR_KL,
    AUTHOR = {Leclerc, Bernard and Thibon, Jean-Yves},
     TITLE = {Littlewood-{R}ichardson coefficients and {K}azhdan-{L}usztig polynomials},
 BOOKTITLE = {Combinatorial methods in representation theory ({K}yoto, 1998)},
    SERIES = {Adv. Stud. Pure Math.},
    VOLUME = {28},
     PAGES = {155--220},
 PUBLISHER = {Kinokuniya, Tokyo},
      YEAR = {2000},
      ISBN = {4-314-10141-5},
   MRCLASS = {20C08 (33D80)},
  MRNUMBER = {1864481},
MRREVIEWER = {Andrew\ Mathas},
       DOI = {10.2969/aspm/02810155},
       URL = {https://doi.org/10.2969/aspm/02810155},
}

@article {Alexandersson2020Lollipop,
    AUTHOR = {Alexandersson, Per},
     TITLE = {L{LT} polynomials, elementary symmetric functions and melting lollipops},
   JOURNAL = {J. Algebraic Combin.},
  FJOURNAL = {Journal of Algebraic Combinatorics. An International Journal},
    VOLUME = {53},
      YEAR = {2021},
    NUMBER = {2},
     PAGES = {299--325},
      ISSN = {0925-9899,1572-9192},
   MRCLASS = {05E05},
  MRNUMBER = {4238181},
MRREVIEWER = {Domenico\ Senato},
       DOI = {10.1007/s10801-019-00929-z},
       URL = {https://doi.org/10.1007/s10801-019-00929-z},
}

@article {Huh2020LLTLollipop,
    AUTHOR = {Huh, JiSun and Nam, Sun-Young and Yoo, Meesue},
     TITLE = {Melting lollipop chromatic quasisymmetric functions and {S}chur expansion of unicellular {LLT} polynomials},
   JOURNAL = {Discrete Math.},
  FJOURNAL = {Discrete Mathematics},
    VOLUME = {343},
      YEAR = {2020},
    NUMBER = {3},
     PAGES = {111728, 21},
      ISSN = {0012-365X,1872-681X},
   MRCLASS = {05E05 (05A15)},
  MRNUMBER = {4033624},
MRREVIEWER = {Houyi\ Yu},
       DOI = {10.1016/j.disc.2019.111728},
       URL = {https://doi.org/10.1016/j.disc.2019.111728},
}

@article {Lee2021LLTlinear,
    AUTHOR = {Lee, Seung Jin},
     TITLE = {Linear relations on {LLT} polynomials and their k-{S}chur positivity for {$k=2$}},
   JOURNAL = {J. Algebraic Combin.},
  FJOURNAL = {Journal of Algebraic Combinatorics. An International Journal},
    VOLUME = {53},
      YEAR = {2021},
    NUMBER = {4},
     PAGES = {973--990},
      ISSN = {0925-9899,1572-9192},
   MRCLASS = {05E05},
  MRNUMBER = {4263640},
MRREVIEWER = {Tanja\ Stojadinovi\'{c}},
       DOI = {10.1007/s10801-020-00950-7},
       URL = {https://doi.org/10.1007/s10801-020-00950-7},
}

@article{DMPS1992hessenbergvarieties,
    AUTHOR = {De Mari, F. and Procesi, C. and Shayman, M. A.},
     TITLE = {Hessenberg varieties},
   JOURNAL = {Trans. Amer. Math. Soc.},
  FJOURNAL = {Transactions of the American Mathematical Society},
    VOLUME = {332},
      YEAR = {1992},
    NUMBER = {2},
     PAGES = {529--534},
      ISSN = {0002-9947,1088-6850},
   MRCLASS = {14L30 (14M17)},
  MRNUMBER = {1043857},
MRREVIEWER = {Michel\ Brion},
       DOI = {10.2307/2154181},
}

@article{SW2016chromaticquasisymmetric,
    AUTHOR = {Shareshian, John and Wachs, Michelle L.},
     TITLE = {Chromatic quasisymmetric functions},
   JOURNAL = {Adv. Math.},
  FJOURNAL = {Advances in Mathematics},
    VOLUME = {295},
      YEAR = {2016},
     PAGES = {497--551},
      ISSN = {0001-8708,1090-2082},
   MRCLASS = {05E05 (14M15)},
  MRNUMBER = {3488041},
MRREVIEWER = {Meesue\ Yoo},
       DOI = {10.1016/j.aim.2015.12.018},
}

@article{ayzenberg2022second,
    AUTHOR = {Ayzenberg, A. A. and Masuda, M. and Sato, T.},
     TITLE = {The second cohomology of regular semisimple {H}essenberg varieties from {GKM} theory},
   JOURNAL = {Tr. Mat. Inst. Steklova},
  FJOURNAL = {Trudy Matematicheskogo Instituta Imeni V. A. Steklova},
    VOLUME = {317},
      YEAR = {2022},
     PAGES = {5--26},
      ISSN = {0371-9685},
   MRCLASS = {57S12 (14M15)},
  MRNUMBER = {4538821},
       DOI = {10.4213/tm4289},
}

@incollection{tymoczko2008permutation,
    AUTHOR = {Tymoczko, Julianna S.},
     TITLE = {Permutation actions on equivariant cohomology of flag varieties},
 BOOKTITLE = {Toric topology},
    SERIES = {Contemp. Math.},
    VOLUME = {460},
     PAGES = {365--384},
 PUBLISHER = {Amer. Math. Soc., Providence, RI},
      YEAR = {2008},
      ISBN = {978-0-8218-4486-1},
   MRCLASS = {14M15 (14N15 20C30 55N91)},
  MRNUMBER = {2428368},
MRREVIEWER = {Christian\ Ohn},
       DOI = {10.1090/conm/460/09030},
}

@article{chohonglee_second_cohom,
    author = {Cho, Soojin and Hong, Jaehyun and Lee, Eunjeong},
    title = {Permutation Module Decomposition of the Second Cohomology of a Regular Semisimple {H}essenberg Variety},
    JOURNAL = {Int. Math. Res. Not.},
    FJOURNAL = {International Mathematics Research Notices},
    year = {2022},
    month = {12},
    issn = {1073-7928},
    doi = {10.1093/imrn/rnac328},
    url = {https://doi.org/10.1093/imrn/rnac328}
}

@misc{chow_linearp,
    author = {Chow, Timothy},
    title = {e-positivity of the coefficient of t in $X_G(t)$},
    url = {http://timothychow.net/h2.pdf}
}

@article{GKM_theory,
    AUTHOR = {Goresky, Mark and Kottwitz, Robert and MacPherson, Robert},
     TITLE = {Equivariant cohomology, {K}oszul duality, and the localization
              theorem},
   JOURNAL = {Invent. Math.},
  FJOURNAL = {Inventiones Mathematicae},
    VOLUME = {131},
      YEAR = {1998},
    NUMBER = {1},
     PAGES = {25--83},
      ISSN = {0020-9910,1432-1297},
   MRCLASS = {55N91 (14F25 14F32 16E99 18G10 55N33)},
  MRNUMBER = {1489894},
MRREVIEWER = {Roy\ Joshua},
       DOI = {10.1007/s002220050197},
}

@misc{GuayPaquet13,
      title={A modular relation for the chromatic symmetric functions of (3+1)-free posets}, 
      author={Mathieu Guay-Paquet},
      year={2013},
      eprint={1306.2400},
      archivePrefix={arXiv},
      primaryClass={math.CO}
}

@article{Gasharov96,
    AUTHOR = {Gasharov, Vesselin},
     TITLE = {Incomparability graphs of {$(3+1)$}-free posets are
              {$s$}-positive},
 BOOKTITLE = {Proceedings of the 6th {C}onference on {F}ormal {P}ower
              {S}eries and {A}lgebraic {C}ombinatorics ({N}ew {B}runswick,
              {NJ}, 1994)},
   JOURNAL = {Discrete Math.},
  FJOURNAL = {Discrete Mathematics},
    VOLUME = {157},
      YEAR = {1996},
    NUMBER = {1-3},
     PAGES = {193--197},
      ISSN = {0012-365X,1872-681X},
   MRCLASS = {05E05 (05C15)},
  MRNUMBER = {1417294},
       DOI = {10.1016/S0012-365X(96)83014-7},
}

@article{harada2017cohomology,
    AUTHOR = {Harada, Megumi and Precup, Martha E.},
     TITLE = {The cohomology of abelian {H}essenberg varieties and the {S}tanley{\textendash}{S}tembridge conjecture},
   JOURNAL = {Algebr. Comb.},
  FJOURNAL = {Algebraic Combinatorics},
    VOLUME = {2},
      YEAR = {2019},
    NUMBER = {6},
     PAGES = {1059--1108},
      ISSN = {2589-5486},
   MRCLASS = {14M15 (05E05)},
  MRNUMBER = {4049838},
MRREVIEWER = {Xin\ Fang},
       DOI = {10.5802/alco.76},
}

@article{Abreu_Nigro20,
    AUTHOR = {Abreu, Alex and Nigro, Antonio},
     TITLE = {Chromatic symmetric functions from the modular law},
   JOURNAL = {J. Combin. Theory Ser. A},
  FJOURNAL = {Journal of Combinatorial Theory. Series A},
    VOLUME = {180},
      YEAR = {2021},
     PAGES = {Paper No. 105407, 30},
      ISSN = {0097-3165,1096-0899},
   MRCLASS = {05E05 (05C15)},
  MRNUMBER = {4199388},
MRREVIEWER = {Houyi\ Yu},
       DOI = {10.1016/j.jcta.2021.105407},
}

@article{brosnan_chow_dotactn_is_chromsym,
    AUTHOR = {Brosnan, Patrick and Chow, Timothy Y.},
     TITLE = {Unit interval orders and the dot action on the cohomology of regular semisimple {H}essenberg varieties},
   JOURNAL = {Adv. Math.},
  FJOURNAL = {Advances in Mathematics},
    VOLUME = {329},
      YEAR = {2018},
     PAGES = {955--1001},
      ISSN = {0001-8708,1090-2082},
   MRCLASS = {05E05 (14M15)},
  MRNUMBER = {3783432},
MRREVIEWER = {Marko\ Radovanovi\'{c}},
       DOI = {10.1016/j.aim.2018.02.020},
}

@misc{guaypaquet2016shar_wachs_conj,
      title={A second proof of the Shareshian--Wachs conjecture, by way of a new Hopf algebra}, 
      author={Mathieu Guay-Paquet},
      year={2016},
      eprint={1601.05498},
      archivePrefix={arXiv},
      primaryClass={math.CO}
}

@article{tymoczko2008schubertreps,
    AUTHOR = {Tymoczko, Julianna S.},
     TITLE = {Permutation representations on {S}chubert varieties},
   JOURNAL = {Amer. J. Math.},
  FJOURNAL = {American Journal of Mathematics},
    VOLUME = {130},
      YEAR = {2008},
    NUMBER = {5},
     PAGES = {1171--1194},
      ISSN = {0002-9327,1080-6377},
   MRCLASS = {14M15 (14F25 14F43)},
  MRNUMBER = {2450205},
MRREVIEWER = {Harry\ Tamvakis},
       DOI = {10.1353/ajm.0.0018},
}

@article{STANLEY1995chromsym,
    AUTHOR = {Stanley, Richard P.},
     TITLE = {A symmetric function generalization of the chromatic polynomial of a graph},
   JOURNAL = {Adv. Math.},
  FJOURNAL = {Advances in Mathematics},
    VOLUME = {111},
      YEAR = {1995},
    NUMBER = {1},
     PAGES = {166--194},
      ISSN = {0001-8708,1090-2082},
   MRCLASS = {05E05 (05C15)},
  MRNUMBER = {1317387},
MRREVIEWER = {SeungKyung\ Park},
       DOI = {10.1006/aima.1995.1020},
}

@misc{Dahlberg19,
      title={Triangular Ladders $P_{d,2}$ are $e$-positive}, 
      author={Samantha Dahlberg},
      year={2019},
      eprint={1811.04885},
      archivePrefix={arXiv},
      primaryClass={math.CO}
}

@article{Ayzenberg2018isospectral,
    AUTHOR = {Ayzenberg, Anton and Buchstaber, Victor},
     TITLE = {Manifolds of isospectral matrices and {H}essenberg varieties},
   JOURNAL = {Int. Math. Res. Not. IMRN},
  FJOURNAL = {International Mathematics Research Notices. IMRN},
      YEAR = {2021},
    VOLUME = {2021},
    NUMBER = {21},
     PAGES = {16671--16692},
      ISSN = {1073-7928,1687-0247},
   MRCLASS = {14M15 (05E14 53E99)},
  MRNUMBER = {4338229},
MRREVIEWER = {Eric\ Hogle},
       DOI = {10.1093/imrn/rnz388},
       URL = {https://doi.org/10.1093/imrn/rnz388},
}

@article{ALEXANDERSSON2018LLTchromsym,
    AUTHOR = {Alexandersson, Per and Panova, Greta},
     TITLE = {L{LT} polynomials, chromatic quasisymmetric functions and graphs with cycles},
   JOURNAL = {Discrete Math.},
  FJOURNAL = {Discrete Mathematics},
    VOLUME = {341},
      YEAR = {2018},
    NUMBER = {12},
     PAGES = {3453--3482},
      ISSN = {0012-365X,1872-681X},
   MRCLASS = {05E05 (05A15)},
  MRNUMBER = {3862644},
MRREVIEWER = {Zhicong\ Lin},
       DOI = {10.1016/j.disc.2018.09.001},
}

@article {ChoHongLee_Bases,
    AUTHOR = {Cho, Soojin and Hong, Jaehyun and Lee, Eunjeong},
     TITLE = {Bases of the equivariant cohomologies of regular semisimple {H}essenberg varieties},
   JOURNAL = {Adv. Math.},
  FJOURNAL = {Advances in Mathematics},
    VOLUME = {423},
      YEAR = {2023},
     PAGES = {Paper No. 109018, 81},
      ISSN = {0001-8708,1090-2082},
   MRCLASS = {14M15 (05E05 14C15 14L30)},
  MRNUMBER = {4577268},
       DOI = {10.1016/j.aim.2023.109018},
       URL = {https://doi.org/10.1016/j.aim.2023.109018},
}

\end{document}